\documentclass[a4paper,11pt, english]{article}
\usepackage[utf8]{inputenc}
\usepackage[T1]{fontenc}
\usepackage{graphicx}
\usepackage[a4paper]{geometry}
\usepackage{amsmath,amsfonts,amssymb,amsthm,epsfig,epstopdf,url,array}
\usepackage{rotating}
\usepackage[colorlinks=true,citecolor=red,linkcolor=blue,pdfpagetransition=Blinds]{hyperref}
\usepackage{cleveref}
\usepackage{nameref}
\usepackage{enumitem}
\usepackage{comment}
\Crefname{paragraph}{Section}{Sections}
\usepackage{fancyhdr}
\usepackage{cases}
\usepackage{mathrsfs}

\usepackage{pgfplots}

\usepackage{fullpage}

\newcommand{\dive}[1]{\mathrm{div}}

\providecommand{\keywords}[1]{\noindent {\textit{Keywords:}} #1}

\usepackage{aliascnt}

\theoremstyle{plain}

\newtheorem{prop}{Proposition}[section]
\newaliascnt{theo}{prop}
\newtheorem{theo}[theo]{Theorem}
\aliascntresetthe{theo}

\newaliascnt{lem}{prop}
\newtheorem{lem}[lem]{Lemma}
\aliascntresetthe{lem}

\newaliascnt{defprop}{prop}

\aliascntresetthe{defprop}

\newaliascnt{cor}{prop}
\newtheorem{cor}[cor]{Corollary}
\aliascntresetthe{cor}

\newaliascnt{rmk}{prop}
\newtheorem{rmk}[rmk]{Remark}
\aliascntresetthe{rmk}

\theoremstyle{definition}

\newaliascnt{defi}{prop}

\aliascntresetthe{defi}

\newaliascnt{app}{prop}

\aliascntresetthe{app}

\newaliascnt{claim}{prop}

\aliascntresetthe{claim}

\newaliascnt{ass}{prop}

\aliascntresetthe{ass}

\crefname{proposition}{Proposition}{Propositions}
\Crefname{proposition}{Proposition}{Propositions}

\crefname{theo}{Theorem}{Theorems}
\Crefname{theo}{Theorem}{Theorems}

\crefname{lemma}{Lemma}{Lemmas}
\Crefname{lemma}{Lemma}{Lemmas}

\crefname{defprop}{Definition--Proposition}{Definition--Propositions}
\Crefname{defprop}{Definition--Proposition}{Definition--Propositions}

\crefname{cor}{Corollary}{Corollaries}
\Crefname{cor}{Corollary}{Corollaries}

\crefname{rmk}{Remark}{Remarks}
\Crefname{rmk}{Remark}{Remarks}

\crefname{defi}{Definition}{Definitions}
\Crefname{defi}{Definition}{Definitions}

\crefname{app}{Application}{Applications}
\Crefname{app}{Application}{Applications}

\crefname{claim}{Claim}{Claims}
\Crefname{claim}{Claim}{Claims}

\crefname{ass}{Assumption}{Assumptions}
\Crefname{ass}{Assumption}{Assumptions}

\def\dx{\,\textnormal{d}x}
\def\dt{\textnormal{d}t}
\def\d{\textnormal{d}}

\newcommand{\vertiii}[1]{{\left\vert\kern-0.25ex\left\vert\kern-0.25ex\left\vert #1 
    \right\vert\kern-0.25ex\right\vert\kern-0.25ex\right\vert}}

\makeatletter
\let\original@addcontentsline\addcontentsline
\newcommand{\dummy@addcontentsline}[3]{}
\newcommand{\DeactivateToc}{\let\addcontentsline\dummy@addcontentsline}
\newcommand{\ActivateToc}{\let\addcontentsline\original@addcontentsline}
\makeatother

\begin{document}

\title{A Lebeau--Robbiano approach to the controllability of the Boussinesq system with a reduced number of controls}
\author{V\'ictor Hern\'andez-Santamar\'ia\thanks{V. Hern\'andez-Santamar\'ia is supported by Project CBF2023-2024-116 of SECIHTHI and by UNAM-DGAPA-PAPIIT grants IA103826, IN117525, and IN102925 (Mexico).}}

\maketitle

\begin{abstract}
In this paper, we prove the local null controllability of the Boussinesq
system in dimensions two and three with a reduced number of localized
controls. More precisely, the controls act on the temperature equation and on
only $N-2$ components of the velocity equation. Our proof is based on a
spectral approach in the spirit of the Lebeau--Robbiano method.

The main difficulty comes from the coupled structure of the system. The
velocity and temperature equations are governed by the Stokes operator and
the Dirichlet Laplacian, respectively. Their natural spectral decompositions
are different, and there is no common spectral localization naturally adapted
to the coupled system. Therefore, the usual componentwise spectral argument
cannot be applied directly. We show that the cascade structure of the
linearized system makes it possible to overcome this difficulty. The main
ingredient is a mixed observability estimate in which only the Stokes
component is spectrally localized, while the heat component is treated
without any frequency restriction. Combining this estimate with the
dissipation of the high Stokes frequencies allows us to carry out a
Lebeau--Robbiano iteration for the linearized system. The resulting linear
controllability estimate is transferred to the nonlinear Boussinesq system
through a time-iteration argument. As an important consequence, the controls
have a small-time cost bounded by $C\exp(C/T)$, which recovers the expected
parabolic order while preserving the reduced number of controls.
\end{abstract}



\footnotesize
\tableofcontents
\normalsize

\section{Introduction}

\subsection{Motivation}

The Boussinesq system is a classical model for the motion of an incompressible
viscous fluid coupled with the evolution of a temperature field. It arises in
the description of buoyancy-driven flows, where variations of temperature
produce density effects that act as a force on the fluid. In its simplest form,
the model brings together two different but closely dissipative
dynamics: the Navier--Stokes equations for the velocity field and a
transport-diffusion equation for the temperature. This coupling appears
naturally in several physical situations, including thermal convection, oceanic
and atmospheric dynamics, and geophysical flows. We refer to
\cite{Maj03,Ped87,Vall17} for further details.

Let $\Omega\subset\mathbb R^N$, with $N\in\{2,3\}$, be a bounded domain with
smooth boundary, and let $T>0$. The Boussinesq system takes the form
\begin{equation}\label{eq:boussinesq-intro}
\left\{
\begin{array}{rcll}
y_t-\nu\Delta y+(y\cdot\nabla)y+\nabla p
&=&
\theta e_N
& \text{in } (0,T)\times\Omega,\\[1mm]
\theta_t-\kappa\Delta\theta+y\cdot\nabla\theta
&=&
0
& \text{in } (0,T)\times\Omega,\\[1mm]
\nabla\cdot y
&=&
0
& \text{in } (0,T)\times\Omega,\\[1mm]
y=0,\quad \theta=0
&&
& \text{on } (0,T)\times\partial\Omega,\\[1mm]
y(0)=y^0,\quad \theta(0)=\theta^0
&&
& \text{in } \Omega.
\end{array}
\right.
\end{equation}
The unknowns are the fluid velocity field
$y=y(t,x)=(y_1(t,x),\ldots,y_N(t,x))$, the pressure $p=p(t,x)$, and the scalar
temperature $\theta=\theta(t,x)$. Here $\nu>0$ is the viscosity,
$\kappa>0$ is the thermal diffusivity, and
\begin{equation*}
e_N:=
\begin{cases}
(0,1), & N=2,\\
(0,0,1), & N=3,
\end{cases}
\end{equation*}
denotes the direction of gravity.

The main feature of \eqref{eq:boussinesq-intro} is its two-way coupled
structure. The temperature acts on the velocity equation through the buoyancy
force $\theta e_N$, while the velocity enters the temperature equation through
the transport term $y\cdot\nabla\theta$. Although both equations have a
parabolic character, they have different analytical structures: the velocity
field is subject to the incompressibility constraint and involves the pressure,
whereas the temperature is governed by a scalar transport-diffusion equation.
At the linear level, this distinction is reflected in the different operators
governing their dissipative dynamics. The interaction between these two
structures will play a central role in the present work.

This interplay between the fluid and temperature dynamics has been the object
of extensive mathematical study. Classical well-posedness results for
Boussinesq-type systems go back at least to \cite{CDB80}. In two space
dimensions, the fully viscous and diffusive system can be treated by adapting
the classical framework for the Navier--Stokes equations; see, for instance,
\cite{Tem88}. A large literature also addresses partially dissipative and
anisotropic variants, as well as models with temperature-dependent viscosity
or diffusivity. We refer to
\cite{HL05,Cha06,DP08,DP11,WZ11,LLT13,LT16} and the references therein.

Our purpose in this paper is to investigate how the coupled dynamics described
above can be influenced by localized external inputs. More precisely, we focus
on the null controllability problem, namely whether the coupled state
$(y,\theta)$ can be driven exactly to zero at a prescribed time $T$ by means of
suitable controls.

To make this precise, let $\omega\subset\Omega$ be an arbitrary nonempty open
set where the controls are localized. We consider the controlled Boussinesq
system
\begin{equation}\label{eq:controlled-boussinesq-intro}
\left\{
\begin{array}{rcll}
y_t-\nu\Delta y+(y\cdot\nabla)y+\nabla p
&=&
\theta e_N+\mathbf 1_\omega v
& \text{in } (0,T)\times\Omega,\\[1mm]
\theta_t-\kappa\Delta\theta+y\cdot\nabla\theta
&=&
\mathbf 1_\omega v_0
& \text{in } (0,T)\times\Omega,\\[1mm]
\nabla\cdot y
&=&
0
& \text{in } (0,T)\times\Omega,\\[1mm]
y=0,\quad \theta=0
&&
& \text{on } (0,T)\times\partial\Omega,\\[1mm]
y(0)=y^0,\quad \theta(0)=\theta^0
&&
& \text{in } \Omega.
\end{array}
\right.
\end{equation}
Here $v_0$ is a scalar control acting in the temperature equation, while
$v=(v_1,\ldots,v_N)$ is a vector-valued control acting in the velocity
equation.

Our goal is to study null controllability for
\eqref{eq:controlled-boussinesq-intro} with a reduced number of controls. More
specifically, we impose two missing components in the velocity control, namely
\begin{equation}\label{eq:missing-controls-intro}
v_{N-1}\equiv 0,
\qquad
v_N\equiv 0,
\end{equation}
and we ask whether the controls can be chosen so that
\begin{equation}\label{eq:null-controllability-intro}
y(T)=0,
\qquad
\theta(T)=0
\qquad
\text{in } \Omega.
\end{equation}
Thus, in dimension $N=2$, the velocity equation is not directly controlled,
whereas in dimension $N=3$ only the first component of the velocity equation is
controlled. In both cases, the temperature equation is controlled through the
scalar input $v_0$.

\subsection{A spectral point of view}

Local controllability for the Boussinesq system, that is, controllability for
small enough initial data, has been studied mainly through Carleman-based
techniques, starting with the work \cite{Gue06}. In particular, the question of
reducing the number of controls has received an affirmative answer in several
settings. For instance, the results in
\cite{FCGIP06,Car12,BPLB22,TdTWZ24} show that one can control the system when
some components of the velocity control are removed, including configurations
closely related to \eqref{eq:missing-controls-intro}.\footnote{%
Another line of work, based on geometric control ideas, addresses global
approximate controllability properties for Boussinesq flows with arbitrary
initial data and a reduced number of controls. We refer to
\cite{NR25,Ris25,Ris26,LXZ25}. These results are different in nature from the
local null controllability problem considered here.}

The corresponding proofs generally separate the linear and nonlinear parts of
the problem. One first establishes null controllability for the system obtained
by linearizing around zero, usually by means of a suitable
Carleman estimate. The nonlinear result is then recovered by treating the
nonlinear terms as perturbations and applying an inverse mapping theorem or a
fixed-point argument. We also begin with the linearized system, but follow a
different construction in both parts of the proof. At the linear level, we
develop a spectral approach in the spirit of the Lebeau--Robbiano method. The
passage to the nonlinear system is then carried out through a quantitative
time-iteration argument.

In the present setting, linearization around the zero trajectory leads to the
controlled Stokes--heat cascade
\begin{equation}\label{eq:linearized-boussinesq-intro}
\left\{
\begin{array}{rcll}
u_t-\nu\Delta u+\nabla q
&=&
\theta e_N+\mathbf 1_\omega v
& \text{in } (0,T)\times\Omega,\\[1mm]
\theta_t-\kappa\Delta\theta
&=&
\mathbf 1_\omega v_0
& \text{in } (0,T)\times\Omega,\\[1mm]
\nabla\cdot u
&=&
0
& \text{in } (0,T)\times\Omega,\\[1mm]
u=0,\quad\theta=0
&&
& \text{on } (0,T)\times\partial\Omega,\\[1mm]
u(0)=u^0,\quad\theta(0)=\theta^0
&&
& \text{in } \Omega.
\end{array}
\right.
\end{equation}
The nonlinear terms do not appear in
\eqref{eq:linearized-boussinesq-intro}, but the coupling produced by the
buoyancy force remains. In particular, the system has a cascade structure:
the temperature evolves according to a heat equation and acts as a source term
in the Stokes equation.

Let $\mathbf A$ denote the Stokes operator, $A$ the Dirichlet Laplacian, and
$\Pi$ the Leray projection onto divergence-free vector fields. At a formal
level, system \eqref{eq:linearized-boussinesq-intro} can be written as
\begin{equation}\label{eq:linearized-boussinesq-abstract-intro}
\left\{
\begin{aligned}
u_t+\nu\mathbf A u
&=
\Pi(\theta e_N)+\Pi(\mathbf 1_\omega v),\\
\theta_t+\kappa A\theta
&=
\mathbf 1_\omega v_0.
\end{aligned}
\right.
\end{equation}
The precise functional framework, together with the well-posedness properties
of the linearized and nonlinear systems, is introduced in
\Cref{sec:functional-setting}.

For a given initial datum $(u^0,\theta^0)$, the null-controllability problem for
\eqref{eq:linearized-boussinesq-intro} consists in finding localized controls
$v_0$ and $v$, satisfying the reduced-control condition
\eqref{eq:missing-controls-intro}, such that the corresponding solution
satisfies
\begin{equation*}
u(T)=0,
\qquad
\theta(T)=0
\qquad
\text{in }\Omega.
\end{equation*}
Unlike the nonlinear problem, the linearized system is not subject to a
smallness condition on the initial datum.

Our first objective is to treat
\eqref{eq:linearized-boussinesq-intro} under the configuration
\eqref{eq:missing-controls-intro} by means of a spectral construction based on
the Lebeau--Robbiano method. For scalar parabolic equations, this method
combines two basic ingredients: a spectral inequality for low frequencies and
the natural dissipation of high frequencies. By alternating these two effects
on a suitable sequence of time intervals, one obtains null controllability.
This strategy goes back to \cite{LR95}. For the Stokes system, an analogous
construction was developed in \cite{CSL16}.

Lebeau--Robbiano-type methods have also been developed for several coupled
systems. We mention, for instance, coupled deterministic parabolic systems
\cite{LZ19}, stochastic coupled parabolic systems \cite{XLiu14,HSP22},
nonlinear reaction-diffusion systems \cite{LeBalch20COCV}, and coupled Stokes
systems \cite{LdT25}. In these settings, the spectral decompositions of the
different components are either the same or can be compared directly. At a
formal level, one can therefore introduce a common low-frequency projection
\begin{equation*}
P_\Lambda U,
\qquad
U=(u_1,\ldots,u_m),
\end{equation*}
and implement the Lebeau--Robbiano strategy on the low-frequency part of the
whole state. Depending on the structure of the coupling, this construction can
also be combined with additional arguments to reduce the number of controls.

The situation for \eqref{eq:linearized-boussinesq-intro} is different. The
velocity component is governed by the Stokes operator, whereas the temperature
component is governed by the Dirichlet Laplacian. Thus, the natural
low-frequency objects are
\begin{equation*}
P_\Lambda^S u,
\qquad
P_\Lambda^D\theta,
\end{equation*}
where $P_\Lambda^S$ and $P_\Lambda^D$ denote the Stokes and Dirichlet spectral
projections, respectively. These projections do not arise from a common
spectral basis for the coupled system.

If controls were allowed to act on the temperature equation and on every
component of the velocity equation, one could use the spectral inequalities
for the heat equation and the Stokes system established in
\cite{LR95} and \cite{CSL16}, respectively, and develop a standard
Lebeau--Robbiano iteration yielding null controllability for
\eqref{eq:linearized-boussinesq-intro}. Some care would still be needed to
handle the two spectral decompositions and the coupling term, but no component
of the state would have to be recovered indirectly through the coupling. The
resulting argument would therefore remain essentially componentwise.

This is no longer the case under condition
\eqref{eq:missing-controls-intro}. Since two velocity components are not
directly controlled, the construction must use the coupled structure of
\eqref{eq:linearized-boussinesq-intro} rather than act independently on each
component. In particular, the buoyancy coupling between the temperature and
the velocity becomes essential. However, this coupling does not preserve the spectral decompositions of the
Stokes operator and the Dirichlet Laplacian. In general, we should not expect
\begin{equation*}
P_\Lambda^S\Pi(\theta e_N)
=
\Pi\bigl((P_\Lambda^D\theta)e_N\bigr).
\end{equation*}
Consequently, a low-frequency temperature component with respect to the
Dirichlet Laplacian may generate Stokes frequencies that are not described by
the same cutoff. This prevents a direct componentwise implementation of the usual Lebeau--Robbiano iteration.

The main idea of this work is to avoid using a common spectral localization
for the two components of \eqref{eq:linearized-boussinesq-intro}. Instead, we
prove a mixed observability estimate for the adjoint Stokes--heat cascade. In
this estimate, only the Stokes terminal datum is spectrally localized, while
the terminal datum of the heat component is allowed to be arbitrary in
$L^2(\Omega)$. This asymmetric localization is natural in view of the
triangular structure of the linearized system and is the key ingredient in our
modified Lebeau--Robbiano iteration. The resulting linear controllability
estimate is then transferred to the nonlinear Boussinesq system through a
time-iteration argument. We describe this strategy in more detail below.

\subsection{Main result}

We now state the main result of this work. In the statement below, $\mathbf V$
denotes the usual space of divergence-free vector fields in
$H^1_0(\Omega)^N$. The precise definitions of the functional spaces used
throughout the paper are given in \Cref{sec:functional-setting}.

\begin{theo}
\label{thm:main-nonlinear-controllability}
For each $N\in\{2,3\}$, there exist $T_0\in(0,1)$ and $C>0$ such that, for
every $T\in(0,T_0)$, there exists $\delta_T>0$ with the following property:
for every initial datum
\begin{equation*}
(y^0,\theta^0)\in\mathbf V\times H^1_0(\Omega),
\qquad
\|y^0\|_{\mathbf V}
+
\|\theta^0\|_{H^1_0(\Omega)}
\leq
\delta_T,
\end{equation*}
there exist controls
\begin{equation*}
v_0\in L^2(0,T;L^2(\omega)),
\qquad
v\in L^2(0,T;L^2(\omega))^N,
\end{equation*}
satisfying
\begin{equation*}
v_{N-1}\equiv0,
\qquad
v_N\equiv0,
\end{equation*}
such that the corresponding controlled solution to
\eqref{eq:controlled-boussinesq-intro} satisfies
\begin{equation*}
y(T)=0,
\qquad
\theta(T)=0
\qquad
\text{in }\Omega.
\end{equation*}
Moreover, the controls satisfy the quantitative estimate
\begin{equation}
\label{eq:main-nonlinear-control-cost}
\|v_0\|_{L^2(0,T;L^2(\omega))}
+
\|v\|_{L^2(0,T;L^2(\omega))^N}
\leq
C
\exp\left(\frac{C}{T}\right)
\left(
\|y^0\|_{\mathbf V}
+
\|\theta^0\|_{H^1_0(\Omega)}
\right).
\end{equation}
\end{theo}

\Cref{thm:main-nonlinear-controllability} establishes local null
controllability for the Boussinesq system with two missing components in the
velocity control. In dimension two, this means that no control acts directly
on the velocity equation, while in dimension three only one velocity
component is directly controlled. Moreover, the control region $\omega$ can
be any nonempty open subset of $\Omega$.

\begin{rmk}
\label{rem:stronger-main-result-2d}
In dimension two, the conclusion of
\Cref{thm:main-nonlinear-controllability} remains valid with
$\mathbf V\times H^1_0(\Omega)$ replaced by
$\mathbf H\times L^2(\Omega)$, both in the smallness condition and in the
estimate of the control cost. The corresponding controlled state is the
unique weak solution of the Boussinesq system. Thus, the regularity used in
the theorem gives a common statement for the two dimensions but is not
necessary in the two-dimensional case. This stronger result is proved in
\Cref{prop:time-iteration-2d}.
\end{rmk}

\begin{rmk}
\label{rem:any-final-time-nonlinear}
The restriction $T\in(0,T_0)$ in
\Cref{thm:main-nonlinear-controllability} is used to describe the small-time
behavior of the control cost \eqref{eq:main-nonlinear-control-cost}. The local
null-controllability result itself holds for every $T>0$. Indeed, if
$T\geq T_0$, one can fix $\tau\in(0,T_0)$ with $\tau<T$, apply the theorem on
$(0,\tau)$, and extend the controls by zero on $(\tau,T)$. Since the
corresponding solution is zero at time $\tau$, it remains identically zero on
$[\tau,T]$ by uniqueness.
\end{rmk}

One important consequence of our construction is the explicit
small-time estimate \eqref{eq:main-nonlinear-control-cost}. For both the heat
equation and the Stokes system, the null-controllability cost is known to be
bounded from above by
\begin{equation*}
C\exp\left(\frac{C}{T}\right),
\end{equation*}
and this exponential order is optimal in those settings, see, e.g.,
\cite{LR95,Mil04,TT07,CSL16}. Thus,
\eqref{eq:main-nonlinear-control-cost} recovers the expected small-time order of the control cost while retaining the reduced number of actions. To the best
of our knowledge, this quantitative estimate has not been obtained previously
for the Boussinesq system.

The existing Carleman-based arguments were not designed to optimize the
dependence of the control cost on the time horizon. Although quantitative
estimates can in principle be extracted from them, directly tracking the
constants produces a larger blow-up as $T\to0^+$. For example, the estimates
in \cite{FCGIP06} lead to a bound of the form $\exp(C/T^4)$ under some
geometric assumptions on the control region. Later works, notably
\cite{Car12}, removed these geometric restrictions, but their purpose was
qualitative rather than the optimization of the small-time cost. In the
related Stackelberg framework of \cite{TdTWZ24}, a direct tracking of the
constants produces a bound of the form $\exp(C/T^{11})$. In this sense,
estimate \eqref{eq:main-nonlinear-control-cost} is one of the main quantitative
features of the approach developed in this paper.

\subsection{Strategy of the proof}

The proof is divided into two main parts. The first one is devoted to the
linearized Boussinesq system and contains the spectral construction used in the
paper. The second one transfers the resulting quantitative controllability
estimate to the nonlinear system without changing the exponential order of the
cost.

We begin with the adjoint system associated with
\eqref{eq:linearized-boussinesq-abstract-intro}. On a time interval
$(0,\tau)$, it is given by
\begin{equation}
\label{eq:adjoint-system-intro}
\begin{cases}
-\varphi_t+\nu\mathbf A\varphi=0
& t\in(0,\tau),\\
-\psi_t+\kappa A\psi=\varphi_N
& t\in(0,\tau),\\
\varphi(\tau)=\varphi_\tau,\qquad \psi(\tau)=\psi_\tau.
\end{cases}
\end{equation}
Here $\varphi_N$ denotes the last component of the Stokes variable
$\varphi$. The reduced number of controls in
\eqref{eq:linearized-boussinesq-intro} is reflected in the adjoint system
through the observation of
\begin{equation*}
\varphi_1,\ldots,\varphi_{N-2},
\qquad
\psi,
\end{equation*}
on $(0,\tau)\times\omega$. Thus, two components of the Stokes variable are not
directly observed.

The second equation in \eqref{eq:adjoint-system-intro} contains
$\varphi_N$ as a source term. A suitable local estimate for this equation
allows us to estimate $\varphi_N$ in terms of the observation of $\psi$. In
dimension two, this gives an estimate for $\varphi_2$ and leaves
$\varphi_1$ unobserved, since there is no direct observation of the Stokes
variable. In dimension three, $\varphi_1$ is directly observed and the scalar
equation gives an estimate for $\varphi_3$, so only $\varphi_2$ remains
unobserved. Thus, in both cases, we are left with exactly one unobserved Stokes
component, namely $\varphi_{N-1}$. We can then apply the spectral inequality
for the Stokes operator with one missing component proved in
\cite{CSFSS25}. Combining this inequality with the local estimate for
$\varphi_N$ gives an estimate for the full Stokes variable in terms of the
available observations.

We next combine this Stokes estimate with an observability estimate for a
nonhomogeneous heat equation in order to estimate the scalar component. This
leads to a mixed short-time observability inequality for
\eqref{eq:adjoint-system-intro}. More precisely, only the terminal datum
$\varphi_\tau$ is assumed to belong to a low-frequency Stokes space, whereas
$\psi_\tau$ is arbitrary in $L^2(\Omega)$. For a Stokes frequency cutoff $M$,
we take $\tau=\frac{1}{\sqrt M}$ and obtain an observability constant bounded
by $C\exp(C\sqrt M)$.

By duality, the mixed observability inequality provides controls for the
forward system \eqref{eq:linearized-boussinesq-abstract-intro} on
$(0,\tau)$ with two properties: the low modes of the velocity are cancelled,
and the entire temperature component is driven to zero. We then set the
controls equal to zero on a second time interval. Since the temperature is
zero at the beginning of this interval, it remains zero during the free
evolution. The buoyancy term in the velocity equation therefore disappears,
and the remaining high Stokes modes decay under the action of the Stokes
semigroup. These two stages form the main building block of our modified
Lebeau--Robbiano iteration.

We repeat these two steps for an increasing sequence of Stokes frequency
levels and a corresponding sequence of shrinking time intervals. At the end of
each control interval, the temperature is again zero and the low Stokes modes
have been cancelled. During the subsequent free evolution, the temperature
remains zero, so the coupling cannot generate new Stokes frequencies. The
iteration can therefore be carried out using only Stokes frequency cutoffs,
without introducing a spectral localization for the heat component or
comparing the two spectral projections. The resulting construction gives null
controllability for the linearized system with cost
$C\exp\left(\frac{C}{T}\right)$.

The second part of the proof transfers this linear estimate to the nonlinear
Boussinesq system. We use a time-iteration argument inspired by the
construction in \cite{Mar26}. This argument can be viewed as a reinterpretation
of the source term method introduced in \cite{LTT13}. On each time interval,
we choose a control that drives the corresponding linearized solution to zero
and apply the same control to the nonlinear system. The state obtained at the
end of the interval is therefore given by the difference between the nonlinear
and linear flows. Energy and interpolation estimates show that this difference
is quadratic with respect to the size of the initial state and the control.

We repeat this construction on a sequence of time intervals whose lengths add
up to $T$. At every step, the quadratic estimate gives a smaller state to which
the next linear control is applied. For sufficiently small initial data, the
resulting sequence converges to zero, and the nonlinear solution reaches zero
at time $T$. The estimates used in the iteration preserve both the reduced
number of controls and the order $\exp(C/T)$ of the linear control cost.

\subsection{Outline}

The paper is organized as follows. In \Cref{sec:functional-setting}, we
introduce the functional framework, the notation for the Stokes and heat
operators, and the well-posedness results used throughout the paper. In
\Cref{sec:spectral}, we prove the spectral estimates and the mixed short-time
observability inequality for the adjoint Stokes--heat cascade. In
\Cref{sec:linear-control}, we develop the modified Lebeau--Robbiano iteration
and obtain the quantitative null-controllability estimate for the linearized
system. Finally, in \Cref{sec:nonlinear}, we combine this quantitative linear
estimate with stability estimates and a time-iteration argument to prove the
local null controllability of the nonlinear Boussinesq system.

\section{Functional setting, notation, and well-posedness results}
\label{sec:functional-setting}

In this section we introduce the functional framework used throughout the
paper. We fix the notation for the spaces and operators associated with the
Boussinesq system, and state some well-posedness results for the linearized and
nonlinear equations that will be used in what follows.

\subsection{Spaces and operators}

We write $\mathbb L^2(\Omega):=L^2(\Omega)^N$,
$\mathbb H^1_0(\Omega):=H^1_0(\Omega)^N$, and
$\mathbb H^2(\Omega):=H^2(\Omega)^N$. The space $\mathbb L^2(\Omega)$ is
endowed with the usual norm
\begin{equation*}
\|u\|_{\mathbb L^2(\Omega)}
=
\left(
\sum_{i=1}^N\|u_i\|_{L^2(\Omega)}^2
\right)^{1/2},
\qquad
u=(u_1,\ldots,u_N).
\end{equation*}
Analogous product norms are used in $\mathbb H^1_0(\Omega)$ and
$\mathbb H^2(\Omega)$. When the domain is $\Omega$ and no confusion is
possible, we simply write $\|\cdot\|_{L^2}$ for both the norms in
$L^2(\Omega)$ and in $\mathbb L^2(\Omega)$. Norms over other subsets, such as
$\omega$ or $\omega_0$, will always be indicated explicitly.

Let
\begin{equation*}
\mathcal V
:=
\left\{
v\in C_c^\infty(\Omega)^N:\ \nabla\cdot v=0
\right\}.
\end{equation*}
We denote by $\mathbf H$ and $\mathbf V$ the closures of $\mathcal V$ in
$\mathbb L^2(\Omega)$ and $\mathbb H^1_0(\Omega)$, respectively. More
precisely,
\begin{equation*}
\mathbf H
=
\left\{
v\in \mathbb L^2(\Omega):
\ \nabla\cdot v=0\ \text{in }\Omega,
\quad
v\cdot n=0\ \text{on }\partial\Omega
\right\},
\end{equation*}
and
\begin{equation*}
\mathbf V=\mathbb H^1_0(\Omega)\cap\mathbf H.
\end{equation*}
The space $\mathbf H$ is endowed with the scalar product inherited from
$\mathbb L^2(\Omega)$, denoted by $(\cdot,\cdot)_{\mathbf H}$. The space
$\mathbf V$ is endowed with the scalar product
\begin{equation*}
(u,v)_{\mathbf V}
:=
\int_\Omega \nabla u:\nabla v\dx,
\qquad
u,v\in\mathbf V,
\end{equation*}
and the associated norm is denoted by $\|\cdot\|_{\mathbf V}$. By Poincaré's
inequality, this norm is equivalent to the usual norm induced by
$\mathbb H^1_0(\Omega)$. We denote by
$\Pi:\mathbb L^2(\Omega)\to\mathbf H$ the Leray projector.

In what follows, we use two different elliptic operators. The scalar
Dirichlet Laplacian is denoted by
\begin{equation*}
A=-\Delta,
\qquad
D(A)=H^2(\Omega)\cap H^1_0(\Omega),
\end{equation*}
as an operator on $L^2(\Omega)$. The Stokes operator is denoted by
\begin{equation*}
\mathbf A=-\Pi\Delta,
\qquad
D(\mathbf A)=\mathbb H^2(\Omega)\cap\mathbf V,
\end{equation*}
as an operator on $\mathbf H$. Both $A$ and $\mathbf A$ are self-adjoint,
positive operators with compact inverse in their corresponding spaces.

Hence, for the scalar operator $A$, there exist a nondecreasing sequence of
positive eigenvalues $(\lambda_j)_{j\in\mathbb N^*}$ and an orthonormal basis
$(\rho_j)_{j\in\mathbb N^*}$ of $L^2(\Omega)$ such that
\begin{equation}
\label{eq:laplacian-eigenproblem}
\begin{cases}
-\Delta\rho_j=\lambda_j\rho_j
& \text{in }\Omega,\\
\rho_j=0
& \text{on }\partial\Omega,
\end{cases}
\qquad
j\in\mathbb N^*,
\end{equation}
with
\begin{equation*}
0<\lambda_1\leq\lambda_2\leq\cdots,
\qquad
\lambda_j\to+\infty.
\end{equation*}
Similarly, for the Stokes operator $\mathbf A$, there exist a nondecreasing
sequence of positive eigenvalues $(\mu_j)_{j\in\mathbb N^*}$ and an
orthonormal basis $(e_j)_{j\in\mathbb N^*}$ of $\mathbf H$ such that
\begin{equation}
\label{eq:stokes-eigenproblem-abstract}
\mathbf A e_j=\mu_j e_j,
\qquad
j\in\mathbb N^*,
\end{equation}
with
\begin{equation*}
0<\mu_1\leq\mu_2\leq\cdots,
\qquad
\mu_j\to+\infty.
\end{equation*}
In other words, each $e_j$ is associated with a pressure $p_j$ and satisfies
\begin{equation*}
\begin{cases}
-\Delta e_j+\nabla p_j=\mu_j e_j
& \text{in }\Omega,\\
\nabla\cdot e_j=0
& \text{in }\Omega,\\
e_j=0
& \text{on }\partial\Omega.
\end{cases}
\end{equation*}

We also introduce the product spaces
\begin{equation*}
\mathbb X^0:=\mathbf H\times L^2(\Omega),
\qquad
\mathbb X^1:=\mathbf V\times H^1_0(\Omega),
\qquad
\mathbb X^2:=D(\mathbf A)\times D(A).
\end{equation*}
They are endowed with their natural norms, namely
\begin{equation*}
\|(u,\theta)\|_{\mathbb X^0}^2
:=
\|u\|_{\mathbf H}^2+\|\theta\|_{L^2}^2,
\qquad
\|(u,\theta)\|_{\mathbb X^1}^2
:=
\|u\|_{\mathbf V}^2+\|\theta\|_{H^1_0}^2,
\end{equation*}
and
\begin{equation*}
\|(u,\theta)\|_{\mathbb X^2}^2
:=
\|\mathbf A u\|_{\mathbf H}^2+\|A\theta\|_{L^2}^2.
\end{equation*}
Note that, by standard elliptic regularity, the norm
$\|\mathbf A\cdot\|_{\mathbf H}$ is equivalent to the
$\mathbb H^2(\Omega)$-norm on $D(\mathbf A)$, and
$\|A\cdot\|_{L^2}$ is equivalent to the $H^2(\Omega)$-norm on $D(A)$.

\subsection{The linearized Stokes--heat cascade}

We next introduce the linearized system that will be used in the
controllability argument. Given source terms $F$ and $f$, we consider
\begin{equation}
\label{eq:linear-abstract}
\begin{cases}
u_t+\nu\mathbf A u
=
\Pi(\theta e_N)+F
& t\in(0,T),\\
\theta_t+\kappa A\theta
=
f
& t\in(0,T),\\
(u(0),\theta(0))=(u_0,\theta_0).
\end{cases}
\end{equation}
We shall use the following well-posedness result.

\begin{theo}[Well-posedness of the linearized system]
\label{thm:linear-wellposedness}
Let $T\in(0,1)$, $F\in L^2(0,T;\mathbf H)$, and
$f\in L^2(0,T;L^2(\Omega))$.

\begin{enumerate}
\item[(i)] If $(u_0,\theta_0)\in\mathbb X^0$, then
\eqref{eq:linear-abstract} admits a unique solution satisfying
\begin{equation*}
(u,\theta)
\in
C([0,T];\mathbb X^0)\cap L^2(0,T;\mathbb X^1).
\end{equation*}
Moreover, there exists a constant $C>0$, depending only on $\Omega$, $\nu$,
and $\kappa$, such that
\begin{equation*}
\begin{aligned}
&\|(u,\theta)\|_{C([0,T];\mathbb X^0)}^2
+
\|(u,\theta)\|_{L^2(0,T;\mathbb X^1)}^2
\\
&\qquad
\leq
C\left(
\|(u_0,\theta_0)\|_{\mathbb X^0}^2
+
\|F\|_{L^2(0,T;\mathbf H)}^2
+
\|f\|_{L^2(0,T;L^2)}^2
\right).
\end{aligned}
\end{equation*}

\item[(ii)] If $(u_0,\theta_0)\in\mathbb X^1$, then
\eqref{eq:linear-abstract} admits a unique solution satisfying
\begin{equation*}
(u,\theta)
\in
C([0,T];\mathbb X^1)\cap L^2(0,T;\mathbb X^2).
\end{equation*}
Moreover, there exists a constant $C>0$, depending only on $\Omega$, $\nu$,
and $\kappa$, such that
\begin{equation*}
\begin{aligned}
&\|(u,\theta)\|_{C([0,T];\mathbb X^1)}^2
+
\|(u,\theta)\|_{L^2(0,T;\mathbb X^2)}^2
\\
&\qquad
\leq
C\left(
\|(u_0,\theta_0)\|_{\mathbb X^1}^2
+
\|F\|_{L^2(0,T;\mathbf H)}^2
+
\|f\|_{L^2(0,T;L^2)}^2
\right).
\end{aligned}
\end{equation*}
\end{enumerate}
\end{theo}

The proof is standard. It follows from a Galerkin approximation based on the
Stokes eigenfunctions in \eqref{eq:stokes-eigenproblem-abstract} and the
Dirichlet Laplacian eigenfunctions in \eqref{eq:laplacian-eigenproblem},
together with usual energy estimates for the Stokes and heat equations.
At some point in the proof, the constants obtained after applying Gronwall's
inequality are of the form $C_1e^{C_2T}$, with $C_1,C_2>0$ depending only on
$\Omega$, $\nu$, and $\kappa$. Since $T\in(0,1)$, these constants can be
bounded uniformly with respect to $T$. For brevity, we omit the details.

\subsection{The nonlinear Boussinesq system}

We now turn to the well-posedness of the nonlinear Boussinesq system. We use
the usual notation
\begin{equation*}
\mathbf B(u,v):=\Pi((u\cdot\nabla)v)
\end{equation*}
for the convective term. With this, we consider
\begin{equation}
\label{eq:nonlinear-boussinesq-strong}
\begin{cases}
u_t+\nu\mathbf A u+\mathbf B(u,u)
=
\Pi(\theta e_N)+F
& t\in(0,T),\\
\theta_t+\kappa A\theta+u\cdot\nabla\theta
=
f
& t\in(0,T),\\
(u(0),\theta(0))=(u_0,\theta_0).
\end{cases}
\end{equation}

In the remainder of this section, we first consider the three-dimensional
case, where a smallness condition is needed to obtain a strong solution
uniformly for $T\in(0,1)$. This is also the case that will require more details
in the nonlinear controllability argument. The two-dimensional case is
globally well posed and is briefly discussed in
\Cref{rem:nonlinear-wellposedness-2d}.

At the $L^2$ level, the usual cancellations are available. More precisely, if
$u\in\mathbf V$ and $\theta\in H^1_0(\Omega)$, then
\begin{equation}\label{eq:cancelations_2d}
\langle\mathbf B(u,u),u\rangle_{\mathbf V',\mathbf V}=0,
\qquad
\int_\Omega (u\cdot\nabla\theta)\theta\dx=0.
\end{equation}
However, in the three-dimensional case below, the nonlinear terms will be
controlled by interpolation and elliptic regularity. We shall use the
following elementary estimates.

\begin{lem}
\label{lem:basic-estimates-3d}
Assume that $N=3$. There exists a constant $C>0$, depending only on $\Omega$,
such that the following estimates hold.

\begin{enumerate}
\item[(i)] $\|u\|_{L^6}\leq C\|u\|_{\mathbf V}$, for every $u\in\mathbf V$.

\item[(ii)] $\|\nabla u\|_{L^3}\leq C\|u\|_{\mathbf V}^{1/2}\|\mathbf A u\|_{\mathbf H}^{1/2}$, for every $u\in D(\mathbf A)$.

\item[(iii)] $\|\nabla\theta\|_{L^3}\leq C\|\theta\|_{H^1_0}^{1/2}\|A\theta\|_{L^2}^{1/2}$, for every $\theta\in D(A)$.
\end{enumerate}
\end{lem}

\begin{proof}
Item \textup{(i)} follows from the Sobolev embedding $H^1_0(\Omega)\hookrightarrow L^6(\Omega)$. For item \textup{(ii)}, we interpolate $\nabla u$ between $L^2(\Omega)$ and $L^6(\Omega)$ and use the elliptic regularity estimate $\|u\|_{H^2}\leq C\|\mathbf A u\|_{\mathbf H}$ for the Stokes operator. The proof of item \textup{(iii)} is analogous, using the elliptic regularity estimate $\|\theta\|_{H^2}\leq C\|A\theta\|_{L^2}$ for the Dirichlet Laplacian.
\end{proof}

We can now state a small-data well-posedness result with constants uniform with respect to $T\in(0,1)$.

\begin{theo}
\label{thm:small-data-strong-wellposedness-3d}
Assume that $N=3$. There exist constants $r_\ast>0$ and $C>0$, depending only
on $\Omega$, $\nu$, and $\kappa$, with the following property: for every
$T\in(0,1)$, every $(u_0,\theta_0)\in\mathbb X^1$,
$F\in L^2(0,T;\mathbf H)$, and $f\in L^2(0,T;L^2(\Omega))$ satisfying
\begin{equation}
\label{eq:smallness-strong-wp-3d}
\|(u_0,\theta_0)\|_{\mathbb X^1}
+
\|F\|_{L^2(0,T;\mathbf H)}
+
\|f\|_{L^2(0,T;L^2)}
\leq r_\ast,
\end{equation}
system \eqref{eq:nonlinear-boussinesq-strong} has a unique solution satisfying
\begin{equation*}
(u,\theta)
\in
C([0,T];\mathbb X^1)\cap L^2(0,T;\mathbb X^2).
\end{equation*}
Moreover,
\begin{equation}
\label{eq:strong-solution-bound-3d}
\begin{aligned}
&\|(u,\theta)\|_{C([0,T];\mathbb X^1)}^2
+
\|(u,\theta)\|_{L^2(0,T;\mathbb X^2)}^2
\\
&\qquad
\leq
C\left(
\|(u_0,\theta_0)\|_{\mathbb X^1}^2
+
\|F\|_{L^2(0,T;\mathbf H)}^2
+
\|f\|_{L^2(0,T;L^2)}^2
\right).
\end{aligned}
\end{equation}
\end{theo}

The proof of \Cref{thm:small-data-strong-wellposedness-3d} follows the
classical strategy used to obtain local strong solutions of the
three-dimensional Navier--Stokes equations; see, e.g.,
\cite{CF88,Tem95}. For \eqref{eq:nonlinear-boussinesq-strong}, the argument
requires some adaptations to handle the buoyancy term and the transport term
in the scalar equation. The existence part can be obtained by a Galerkin
approximation, first constructing approximate solutions and then using the
corresponding uniform estimates to pass to the limit. Once these estimates
are available, a standard argument gives uniqueness. For brevity, we omit
this part.

Therefore, the proof below only focuses on obtaining estimate
\eqref{eq:strong-solution-bound-3d}. The main point is to derive it in a form
that yields constants $C$ and $r_\ast$ independent of $T\in(0,1)$.

\begin{proof}[Proof of \Cref{thm:small-data-strong-wellposedness-3d}]
We multiply the velocity equation in
\eqref{eq:nonlinear-boussinesq-strong} by $\mathbf A u$ in $\mathbf H$. Then,
by integrating by parts,
\begin{equation}
\label{eq:strong-u-energy-boussinesq}
\frac12\frac{\d}{\dt}\|u(t)\|_{\mathbf V}^2
+
\nu\|\mathbf A u(t)\|_{\mathbf H}^2
=
-\bigl(\mathbf B(u,u),\mathbf A u\bigr)_{\mathbf H}
+
\bigl(\Pi(\theta e_3),\mathbf A u\bigr)_{\mathbf H}
+
(F,\mathbf A u)_{\mathbf H}.
\end{equation}
Let us estimate the three terms on the right-hand side of
\eqref{eq:strong-u-energy-boussinesq}. Using items \textup{(i)} and
\textup{(ii)} of \Cref{lem:basic-estimates-3d} together with Young's
inequality, we bound the first one as
\begin{equation*}
\begin{aligned}
\left|
\bigl(\mathbf B(u,u),\mathbf A u\bigr)_{\mathbf H}
\right|
&\leq
C\|u\|_{L^6}
\|\nabla u\|_{L^3}
\|\mathbf A u\|_{\mathbf H}
\\
&\leq
C\|u\|_{\mathbf V}^{3/2}
\|\mathbf A u\|_{\mathbf H}^{3/2}
\\
&\leq
\frac{\nu}{4}\|\mathbf A u\|_{\mathbf H}^2
+
C\|u\|_{\mathbf V}^6.
\end{aligned}
\end{equation*}
For the second and third terms, using the Cauchy--Schwarz and Young
inequalities, we have
\begin{equation*}
\left|
\bigl(\Pi(\theta e_3),\mathbf A u\bigr)_{\mathbf H}
\right|
+
\left|
(F,\mathbf A u)_{\mathbf H}
\right|
\leq
\frac{\nu}{4}\|\mathbf A u\|_{\mathbf H}^2
+
C\|\theta\|_{L^2}^2
+
C\|F\|_{\mathbf H}^2.
\end{equation*}
Therefore, from \eqref{eq:strong-u-energy-boussinesq}, we deduce
\begin{equation}
\label{eq:u-strong-diff-ineq}
\frac{\d}{\dt}\|u(t)\|_{\mathbf V}^2
+
\nu\|\mathbf A u(t)\|_{\mathbf H}^2
\leq
C\|u(t)\|_{\mathbf V}^6
+
C\|\theta(t)\|_{L^2}^2
+
C\|F(t)\|_{\mathbf H}^2.
\end{equation}

We now multiply the scalar equation in
\eqref{eq:nonlinear-boussinesq-strong} by $A\theta$ in $L^2(\Omega)$. Then
\begin{equation}
\label{eq:strong-theta-energy-boussinesq}
\frac12\frac{\d}{\dt}\|\theta(t)\|_{H^1_0}^2
+
\kappa\|A\theta(t)\|_{L^2}^2
=
-\int_\Omega (u\cdot\nabla\theta)A\theta\dx
+
(f,A\theta)_{L^2}.
\end{equation}
We estimate the two terms on the right-hand side of
\eqref{eq:strong-theta-energy-boussinesq}. Using items \textup{(i)} and
\textup{(iii)} of \Cref{lem:basic-estimates-3d}, we get
\begin{equation*}
\begin{aligned}
\left|
\int_\Omega (u\cdot\nabla\theta)A\theta\dx
\right|
&\leq
C\|u\|_{L^6}
\|\nabla\theta\|_{L^3}
\|A\theta\|_{L^2}
\\
&\leq
C\|u\|_{\mathbf V}
\|\theta\|_{H^1_0}^{1/2}
\|A\theta\|_{L^2}^{3/2}
\\
&\leq
\frac{\kappa}{4}\|A\theta\|_{L^2}^2
+
C\|u\|_{\mathbf V}^4
\|\theta\|_{H^1_0}^2.
\end{aligned}
\end{equation*}
For the second term, using the Cauchy--Schwarz and Young inequalities, we have
\begin{equation*}
\left|
(f,A\theta)_{L^2}
\right|
\leq
\frac{\kappa}{4}\|A\theta\|_{L^2}^2
+
C\|f\|_{L^2}^2.
\end{equation*}
Thus, from \eqref{eq:strong-theta-energy-boussinesq}, we obtain
\begin{equation}
\label{eq:theta-strong-diff-ineq}
\frac{\d}{\dt}\|\theta(t)\|_{H^1_0}^2
+
\kappa\|A\theta(t)\|_{L^2}^2
\leq
C\|u(t)\|_{\mathbf V}^4
\|\theta(t)\|_{H^1_0}^2
+
C\|f(t)\|_{L^2}^2.
\end{equation}

Set
\begin{equation*}
\mathcal Y(t)
:=
\|u(t)\|_{\mathbf V}^2
+
\|\theta(t)\|_{H^1_0}^2,
\qquad
\mathcal Z(t)
:=
\|\mathbf A u(t)\|_{\mathbf H}^2
+
\|A\theta(t)\|_{L^2}^2.
\end{equation*}
Adding \eqref{eq:u-strong-diff-ineq} and
\eqref{eq:theta-strong-diff-ineq}, and using Poincaré's inequality, we deduce
that, for some constants $c_0>0$ and $C>0$,
\begin{equation}
\label{eq:strong-ode-boussinesq}
\frac{\d}{\dt}\mathcal Y(t)
+
c_0\mathcal Z(t)
\leq
C\mathcal Y(t)
+
C\mathcal Y(t)^3
+
C\|F(t)\|_{\mathbf H}^2
+
C\|f(t)\|_{L^2}^2.
\end{equation}

We now use the smallness of the initial datum and the source terms. Let
$r\in(0,1)$ be fixed for the moment and assume that
\begin{equation*}
\|(u_0,\theta_0)\|_{\mathbb X^1}
+
\|F\|_{L^2(0,T;\mathbf H)}
+
\|f\|_{L^2(0,T;L^2)}
\leq r.
\end{equation*}
Define
\begin{equation*}
E_0
:=
\mathcal Y(0)
+
\|F\|_{L^2(0,T;\mathbf H)}^2
+
\|f\|_{L^2(0,T;L^2)}^2.
\end{equation*}
Then $E_0\leq r^2$.

Let
\begin{equation*}
T^\sharp
:=
\sup\left\{
t\in[0,T]:
\mathcal Y(s)\leq1
\ \text{for every }s\in[0,t]
\right\}.
\end{equation*}
Since $\mathcal Y(0)\leq r^2<1$ and $\mathcal Y$ is continuous, the set
defining $T^\sharp$ is nonempty. On $[0,T^\sharp]$, we have
$\mathcal Y^3\leq\mathcal Y$. Hence,
\eqref{eq:strong-ode-boussinesq} gives
\begin{equation*}
\frac{\d}{\dt}\mathcal Y(t)
\leq
C\mathcal Y(t)
+
C\|F(t)\|_{\mathbf H}^2
+
C\|f(t)\|_{L^2}^2.
\end{equation*}
Since $T\in(0,1)$, Gronwall's inequality yields, for every
$t\in[0,T^\sharp]$,
\begin{equation}
\label{eq:bootstrap-Y-bound}
\mathcal Y(t)
\leq
CE_0
\leq
Cr^2,
\end{equation}
where $C>0$ depends only on $\Omega$, $\nu$, and $\kappa$.

We now choose $r_\ast\in(0,1)$, depending only on $\Omega$, $\nu$, and
$\kappa$, small enough so that
\begin{equation*}
Cr_\ast^2\leq\frac12.
\end{equation*}
If $r\leq r_\ast$, then \eqref{eq:bootstrap-Y-bound} gives
$\mathcal Y(t)\leq\frac12$ for every $t\in[0,T^\sharp]$. We claim that
$T^\sharp=T$. Indeed, assume by contradiction that $T^\sharp<T$. Since
$\mathcal Y(T^\sharp)\leq\frac12$, by continuity of $\mathcal Y$ there exists
$\varepsilon>0$ such that $T^\sharp+\varepsilon\leq T$ and
$\mathcal Y(t)\leq1$ for every
$t\in[T^\sharp,T^\sharp+\varepsilon]$. Together with the definition of
$T^\sharp$, this gives $\mathcal Y(t)\leq1$ for every
$t\in[0,T^\sharp+\varepsilon]$, which contradicts the definition of
$T^\sharp$. Hence $T^\sharp=T$. Consequently,
\begin{equation}
\label{eq:Y-bound-small-data}
\sup_{0\leq t\leq T}\mathcal Y(t)
\leq
CE_0.
\end{equation}

Integrating \eqref{eq:strong-ode-boussinesq} over $(0,T)$ and using
\eqref{eq:Y-bound-small-data} and $E_0\leq r_\ast^2$, we obtain
\begin{equation}
\label{eq:Z-bound-small-data}
\int_0^T\mathcal Z(t)\dt
\leq
CE_0.
\end{equation}
Combining \eqref{eq:Y-bound-small-data} and
\eqref{eq:Z-bound-small-data}, we get
\begin{equation*}
\|(u,\theta)\|_{C([0,T];\mathbb X^1)}^2
+
\|(u,\theta)\|_{L^2(0,T;\mathbb X^2)}^2
\leq
CE_0.
\end{equation*}
By the definition of $E_0$, this gives
\eqref{eq:strong-solution-bound-3d}.
\end{proof}

\begin{rmk}
\label{rem:nonlinear-wellposedness-2d}
When $N=2$, system \eqref{eq:nonlinear-boussinesq-strong} is globally well
posed both in $\mathbb X^0$ and in $\mathbb X^1$. More precisely, for every
$T>0$, $(u_0,\theta_0)\in\mathbb X^0$,
$F\in L^2(0,T;\mathbf H)$, and $f\in L^2(0,T;L^2(\Omega))$, there exists a
unique weak solution satisfying
\begin{equation*}
(u,\theta)
\in
C([0,T];\mathbb X^0)\cap L^2(0,T;\mathbb X^1).
\end{equation*}
If, in addition, $(u_0,\theta_0)\in\mathbb X^1$, then the solution satisfies
\begin{equation*}
(u,\theta)
\in
C([0,T];\mathbb X^1)\cap L^2(0,T;\mathbb X^2).
\end{equation*}
Thus, unlike the three-dimensional case, no smallness condition is needed for
the existence and uniqueness of either weak or strong solutions.

Indeed, the cancellations in \eqref{eq:cancelations_2d} give the usual
$L^2$ energy estimate and ensure that the corresponding weak solution
satisfies
\begin{equation}
\label{eq:weak-solution-regularity-2d}
(u,\theta)
\in
L^\infty(0,T;\mathbb X^0)\cap L^2(0,T;\mathbb X^1).
\end{equation}
To obtain the additional regularity, we use the two-dimensional interpolation
inequalities
\begin{equation*}
\|u\|_{L^4}
\leq
C\|u\|_{\mathbf H}^{1/2}\|u\|_{\mathbf V}^{1/2},
\qquad
\|\nabla u\|_{L^4}
\leq
C\|u\|_{\mathbf V}^{1/2}
\|\mathbf A u\|_{\mathbf H}^{1/2},
\end{equation*}
and
\begin{equation*}
\|\nabla\theta\|_{L^4}
\leq
C\|\theta\|_{H^1_0}^{1/2}
\|A\theta\|_{L^2}^{1/2}.
\end{equation*}
These estimates allow us to control the nonlinear term in the velocity
equation as follows
\begin{equation*}
\begin{aligned}
\left|
\bigl(\mathbf B(u,u),\mathbf A u\bigr)_{\mathbf H}
\right|
&\leq
C\|u\|_{L^4}
\|\nabla u\|_{L^4}
\|\mathbf A u\|_{\mathbf H}
\\
&\leq
C\|u\|_{\mathbf H}^{1/2}
\|u\|_{\mathbf V}
\|\mathbf A u\|_{\mathbf H}^{3/2}
\\
&\leq
\frac{\nu}{4}\|\mathbf A u\|_{\mathbf H}^2
+
C\|u\|_{\mathbf H}^2\|u\|_{\mathbf V}^4.
\end{aligned}
\end{equation*}
Similarly, for the transport term in the temperature equation, we obtain
\begin{equation*}
\begin{aligned}
\left|
\int_\Omega (u\cdot\nabla\theta)A\theta\dx
\right|
&\leq
C\|u\|_{L^4}
\|\nabla\theta\|_{L^4}
\|A\theta\|_{L^2}
\\
&\leq
C\|u\|_{\mathbf H}^{1/2}
\|u\|_{\mathbf V}^{1/2}
\|\theta\|_{H^1_0}^{1/2}
\|A\theta\|_{L^2}^{3/2}
\\
&\leq
\frac{\kappa}{4}\|A\theta\|_{L^2}^2
+
C\|u\|_{\mathbf H}^2
\|u\|_{\mathbf V}^2
\|\theta\|_{H^1_0}^2.
\end{aligned}
\end{equation*}

Setting $\mathcal Y(t)
:=
\|u(t)\|_{\mathbf V}^2
+
\|\theta(t)\|_{H^1_0}^2$, and proceeding as in the proof of
\Cref{thm:small-data-strong-wellposedness-3d}, we deduce
\begin{equation*}
\begin{aligned}
\frac{\d}{\dt}\mathcal Y(t)
&+
c_0\left(
\|\mathbf A u(t)\|_{\mathbf H}^2
+
\|A\theta(t)\|_{L^2}^2
\right)
\\
&\leq
C\left(
1+\|u(t)\|_{\mathbf H}^2\|u(t)\|_{\mathbf V}^2
\right)\mathcal Y(t)
+
C\|F(t)\|_{\mathbf H}^2
+
C\|f(t)\|_{L^2}^2.
\end{aligned}
\end{equation*}
The regularity in \eqref{eq:weak-solution-regularity-2d} implies that the
coefficient multiplying $\mathcal Y$ belongs to $L^1(0,T)$. More precisely,
\begin{equation*}
\begin{aligned}
\int_0^T
\left(
1+\|u(t)\|_{\mathbf H}^2\|u(t)\|_{\mathbf V}^2
\right)\dt
&\leq
T
+
\|u\|_{L^\infty(0,T;\mathbf H)}^2
\|u\|_{L^2(0,T;\mathbf V)}^2 <+\infty.
\end{aligned}
\end{equation*}
Gronwall's inequality then yields the global $H^1$-estimate and the additional
regularity $(u,\theta)
\in
C([0,T];\mathbb X^1)\cap L^2(0,T;\mathbb X^2)$. Finally, uniqueness follows from the corresponding energy estimate for the
difference of two solutions.
\end{rmk}

\section{A mixed observability estimate}
\label{sec:spectral}

The goal of this section is to prove an observability estimate for the adjoint
system to \eqref{eq:linear-abstract} on a short time interval of length
$\tau>0$, to be chosen later. The estimate is mixed in the following sense:
the terminal datum of the Stokes adjoint equation is assumed to belong to a
finite-dimensional spectral space, whereas the terminal datum of the scalar
adjoint equation is allowed to be arbitrary in $L^2(\Omega)$.

More precisely, we consider the backward system
\begin{equation}
\label{eq:adjoint-system}
\begin{cases}
-\varphi_t+\nu\mathbf A\varphi=0
& t\in(0,\tau),\\
-\psi_t+\kappa A\psi=\varphi_N
& t\in(0,\tau),\\
\varphi(\tau)=\varphi_\tau,\qquad \psi(\tau)=\psi_\tau,
\end{cases}
\end{equation}
where $\varphi_N$ denotes the last component of the vector field
$\varphi=(\varphi_1,\ldots,\varphi_N)$.

Recalling the spectral decomposition
\eqref{eq:stokes-eigenproblem-abstract}, for $M>0$ we set
\begin{equation*}
\mathbf E_M
:=
\operatorname{span}\{e_j:\mu_j\leq M\},
\end{equation*}
and denote by $\mathbf P_M:\mathbf H\to\mathbf E_M$ the corresponding
orthogonal projection. Throughout this section, the terminal data in
\eqref{eq:adjoint-system} are assumed to satisfy
\begin{equation*}
\varphi_\tau\in\mathbf E_M,
\qquad
\psi_\tau\in L^2(\Omega).
\end{equation*}
Thus, no spectral restriction is imposed on the scalar terminal datum. This
is the main difference with the usual finite-dimensional construction for
coupled parabolic systems, where the different components are localized in
low-frequency spaces. Here only the Stokes component is spectrally
localized, while the scalar component will be treated by means of an
observability estimate for the nonhomogeneous heat equation.

We first recall the spectral inequality for the Stokes operator that will be
used in the proof.

\begin{prop}
\label{prop:stokes-spectral-inequality}
Let $\omega\subset\Omega$ be a nonempty open subset. There exist constants
$C>0$, $C_S>0$, and $M_0>0$ such that, for every $M\geq M_0$, every
$m\in\{1,\ldots,N\}$, and every
$v=(v_1,\ldots,v_N)\in\mathbf E_M$, one has
\begin{equation}
\label{eq:stokes-spectral-inequality}
\|v\|_{\mathbf H}^2
\leq
C e^{C_S\sqrt M}
\int_\omega
\sum_{\substack{1\leq i\leq N\\ i\neq m}}
|v_i(x)|^2\dx.
\end{equation}
\end{prop}

The estimate in \Cref{prop:stokes-spectral-inequality} was proved in
\cite{CSFSS25} with the last component omitted from the observation. As
pointed out in Remark~1.2 therein, this choice is only made for notational
convenience: any one of the components may be omitted without changing the
estimate of the observability cost. We shall use the inequality with
$m=N-1$.

We can now state the main result of this section.

\begin{theo}
\label{thm:mixed-observability}
There exist constants $C>0$ and $M_\ast\geq1$ such that, for every
$M\geq M_\ast$, if $\tau=1/\sqrt M$, then the solution $(\varphi,\psi)$ to
\eqref{eq:adjoint-system} with terminal data
\begin{equation*}
(\varphi_\tau,\psi_\tau)
\in
\mathbf E_M\times L^2(\Omega)
\end{equation*}
satisfies
\begin{equation}
\label{eq:mixed-observability}
\|\varphi(0)\|_{\mathbf H}^2
+
\|\psi(0)\|_{L^2}^2
\leq
C e^{C\sqrt M}
\int_0^\tau\int_\omega
\left(
\sum_{i=1}^{N-2}|\varphi_i|^2+|\psi|^2
\right)\dx\dt.
\end{equation}
When $N=2$, the sum involving the components of $\varphi$ is understood to be
zero.
\end{theo}

For readability, we divide the proof of
\Cref{thm:mixed-observability} into several steps, each presented in a
separate subsection. We first derive some consequences of the spectral
localization of the Stokes terminal datum. We then estimate $\varphi_N$ by
using the scalar adjoint equation and combine this estimate with
\Cref{prop:stokes-spectral-inequality}. Finally, an observability estimate for
the nonhomogeneous heat equation allows us to estimate the scalar component
and conclude \eqref{eq:mixed-observability}.

\subsection{Consequences of the spectral localization for the Stokes equation}

Let $\tau>0$ be fixed. We begin with some elementary consequences of the
condition $\varphi_\tau\in\mathbf E_M$. More precisely, if
\begin{equation}
\label{eq:final_data_fds}
\varphi_\tau
=
\sum_{\mu_j\leq M}a_j e_j,
\end{equation}
then the solution of
\begin{equation}
\label{eq:back-stokes}
-\varphi_t+\nu\mathbf A\varphi=0,
\qquad
\varphi(\tau)=\varphi_\tau,
\end{equation}
is given explicitly by
\begin{equation}
\label{eq:rep_formula_stokes_fd}
\varphi(t)
=
\sum_{\mu_j\leq M}
a_j e^{-\nu\mu_j(\tau-t)}e_j,
\qquad
t\in[0,\tau].
\end{equation}
In particular, $\varphi(t)\in\mathbf E_M$ for every $t\in[0,\tau]$.

We shall repeatedly use the following estimates.

\begin{lem}
\label{lem:stokes-spectral-regularity}
Let $M\geq1$ and let $\varphi$ be a solution of \eqref{eq:back-stokes} with
$\varphi_\tau\in\mathbf E_M$. Then:
\begin{enumerate}[label=({\roman*})]
\item \label{eq:phi-energy-decay}
$\|\varphi(t)\|_{\mathbf H}\leq\|\varphi(s)\|_{\mathbf H}$ for all
$0\leq t\leq s\leq\tau$.

\item \label{eq:stokes-Aphi-bound}
$\|\mathbf A\varphi(t)\|_{\mathbf H}
\leq M\|\varphi(t)\|_{\mathbf H}$ for all $0\leq t\leq\tau$.

\item \label{eq:stokes-time-bound}
There exists a constant $C>0$, depending only on $\nu$ and $\Omega$, such
that
\begin{equation*}
\|\varphi_t(t)\|_{\mathbf H}
+
\|\varphi(t)\|_{\mathbb H^2}
\leq
CM\|\varphi(t)\|_{\mathbf H}
\end{equation*}
for all $0\leq t\leq\tau$.
\end{enumerate}
\end{lem}

\begin{proof}
From the representation formula \eqref{eq:rep_formula_stokes_fd} and the
orthogonality of the eigenfunctions $e_j$, we have, for
$0\leq t\leq s\leq\tau$,
\begin{equation*}
\begin{aligned}
\|\varphi(t)\|_{\mathbf H}^2=
\sum_{\mu_j\leq M}
|a_j|^2e^{-2\nu\mu_j(\tau-t)}&\leq
\sum_{\mu_j\leq M}
|a_j|^2e^{-2\nu\mu_j(\tau-s)}
=
\|\varphi(s)\|_{\mathbf H}^2.
\end{aligned}
\end{equation*}
This proves \ref{eq:phi-energy-decay}. Similarly,
\begin{equation*}
\begin{aligned}
\|\mathbf A\varphi(t)\|_{\mathbf H}^2
&=
\sum_{\mu_j\leq M}
\mu_j^2|a_j|^2e^{-2\nu\mu_j(\tau-t)}
\\
&\leq
M^2
\sum_{\mu_j\leq M}
|a_j|^2e^{-2\nu\mu_j(\tau-t)}
=
M^2\|\varphi(t)\|_{\mathbf H}^2,
\end{aligned}
\end{equation*}
which proves \ref{eq:stokes-Aphi-bound}. From \eqref{eq:back-stokes}, we have
$\varphi_t=\nu\mathbf A\varphi$, and therefore
\begin{equation*}
\|\varphi_t(t)\|_{\mathbf H}
\leq
\nu M\|\varphi(t)\|_{\mathbf H}.
\end{equation*}
On the other hand, elliptic regularity for the Stokes operator and item
\ref{eq:stokes-Aphi-bound} give
\begin{equation*}
\|\varphi(t)\|_{\mathbb H^2}
\leq
C\|\mathbf A\varphi(t)\|_{\mathbf H}
\leq
CM\|\varphi(t)\|_{\mathbf H}.
\end{equation*}
Combining the last two estimates proves \ref{eq:stokes-time-bound}.
\end{proof}

As a direct consequence of \Cref{lem:stokes-spectral-regularity}, for every
$i\in\{1,\ldots,N\}$ and every $t\in[0,\tau]$, we also have
\begin{equation}
\label{eq:component-H2-estimate}
\begin{aligned}
\|\partial_t\varphi_i(t)\|_{L^2}
+
\|\nabla\varphi_i(t)\|_{L^2}
+
\|D^2\varphi_i(t)\|_{L^2}
&\leq
\|\varphi_t(t)\|_{\mathbf H}
+
\|\varphi(t)\|_{\mathbb H^2}
\\
&\leq
CM\|\varphi(t)\|_{\mathbf H}.
\end{aligned}
\end{equation}

We shall also need the following comparison of the solution on different time
intervals.

\begin{lem}
\label{lem:stokes-time-comparison}
Let $M\geq1$ and let $\varphi$ be a solution of \eqref{eq:back-stokes} with
$\varphi_\tau\in\mathbf E_M$. Then, for every $t,s\in[0,\tau]$, we have
\begin{equation}
\label{eq:stokes-time-comparison}
\|\varphi(t)\|_{\mathbf H}^2
\leq
e^{2\nu M|t-s|}
\|\varphi(s)\|_{\mathbf H}^2.
\end{equation}
Moreover, for any two nonempty intervals $J,I\subset(0,\tau)$,
\begin{equation}
\label{eq:stokes-integral-comparison}
\int_J\|\varphi(t)\|_{\mathbf H}^2\dt
\leq
\frac{|J|}{|I|}
e^{2\nu M\tau}
\int_I\|\varphi(t)\|_{\mathbf H}^2\dt.
\end{equation}
\end{lem}

\begin{proof}
From \eqref{eq:rep_formula_stokes_fd}, for every $t,s\in[0,\tau]$ we have
\begin{equation*}
\varphi(t)
=
\sum_{\mu_j\leq M}
a_j e^{-\nu\mu_j(\tau-s)}
e^{\nu\mu_j(t-s)}e_j.
\end{equation*}
Using the orthogonality of the eigenfunctions and the bound $\mu_j\leq M$,
we obtain
\begin{equation*}
\begin{aligned}
\|\varphi(t)\|_{\mathbf H}^2
&=
\sum_{\mu_j\leq M}
|a_j|^2e^{-2\nu\mu_j(\tau-s)}
e^{2\nu\mu_j(t-s)} \leq
e^{2\nu M|t-s|}
\|\varphi(s)\|_{\mathbf H}^2.
\end{aligned}
\end{equation*}
This proves \eqref{eq:stokes-time-comparison}. We now integrate
\eqref{eq:stokes-time-comparison} with respect to $s$ over $I$. Since
$|t-s|\leq\tau$ for $t,s\in[0,\tau]$, for every $t\in J$ we obtain
\begin{equation*}
|I|\|\varphi(t)\|_{\mathbf H}^2
\leq
e^{2\nu M\tau}
\int_I\|\varphi(s)\|_{\mathbf H}^2\,ds.
\end{equation*}
Integrating this inequality with respect to $t$ over $J$ gives
\eqref{eq:stokes-integral-comparison}.
\end{proof}

%

\subsection{A local estimate for the fluid coupling}

In this subsection, we assume that $\tau\in(0,1)$ is fixed. The goal is to
show that the scalar equation for $\psi$ allows us to recover the component
$\varphi_N$ locally in space and time.

For this, let $\omega_0\Subset\omega$ be a nonempty open subset. We choose a
function $\chi\in C_c^\infty(\omega)$ such that
\begin{equation*}
0\leq\chi\leq1,
\qquad
\chi\equiv1
\quad\text{in }\omega_0.
\end{equation*}
We also fix a time cut-off $\xi\in C_c^\infty(0,\tau)$ such that
\begin{equation}
\label{eq:time-cutoff}
0\leq\xi\leq1,
\qquad
\xi\equiv1
\quad\text{in }I_\tau,
\qquad
\left\|
\frac{|\xi'|^2}{\xi}
\right\|_{L^\infty(0,\tau)}
\leq
\frac{C}{\tau^2},
\end{equation}
where
\begin{equation}
\label{eq:def_Itau}
I_\tau
:=
\left(
\frac{\tau}{4},\frac{3\tau}{4}
\right).
\end{equation}
Indeed, such a function can be obtained by taking some
$\eta_0\in C_c^\infty(0,1)$ with $0\leq\eta_0\leq1$ and
$\eta_0\equiv1$ in $(\tfrac14,\tfrac34)$, and setting
$\xi(t)=\eta_0(t/\tau)^2$.

\begin{lem}
\label{lem:local-fluid-coupling}
Let $M\geq1$, $\tau\in(0,1)$, and let $(\varphi,\psi)$ be the solution of
\eqref{eq:adjoint-system} with $\varphi_\tau\in\mathbf E_M$. There exists a
constant $C>0$, depending only on $\Omega$, $\omega_0$, $\omega$, $\nu$, and
$\kappa$, but independent of $M$, $\tau$, $\varphi$, and $\psi$, such that
\begin{equation}
\label{eq:local-fluid-coupling}
\int_{I_\tau}\int_{\omega_0}|\varphi_N|^2\dx\dt
\leq
C\left(
M+\frac1\tau
\right)
\left(
\int_0^\tau\int_\omega|\psi|^2\dx\dt
\right)^{1/2}
\left(
\int_0^\tau\|\varphi(t)\|_{\mathbf H}^2\dt
\right)^{1/2}.
\end{equation}
\end{lem}

\begin{proof}
Since $\xi\equiv1$ in $I_\tau$ and $\chi\equiv1$ in $\omega_0$, we readily
have
\begin{equation}
\label{eq:localized-left-phiN}
\int_{I_\tau}\int_{\omega_0}|\varphi_N|^2\dx\dt
\leq
\int_0^\tau\int_\Omega
\xi\chi^2|\varphi_N|^2\dx\dt.
\end{equation}
On the other hand, using the scalar equation in
\eqref{eq:adjoint-system}, we deduce that
\begin{equation*}
\int_0^\tau\int_\Omega
\xi\chi^2|\varphi_N|^2\dx\dt
=
\int_0^\tau\int_\Omega
\left(
-\psi_t-\kappa\Delta\psi
\right)
\xi\chi^2\varphi_N\dx\dt.
\end{equation*}
Integrating by parts in time and in space gives
\begin{equation}
\label{eq:ibp-local-fluid-coupling}
\int_0^\tau\int_\Omega
\xi\chi^2|\varphi_N|^2\dx\dt
=
\int_0^\tau\int_\Omega
\psi
\left[
\partial_t\left(\xi\chi^2\varphi_N\right)
-
\kappa\Delta\left(\xi\chi^2\varphi_N\right)
\right]\dx\dt.
\end{equation}
Notice that no boundary terms appear since the time cut-off $\xi$ is
compactly supported in $(0,\tau)$, while the spatial cut-off $\chi$ is
compactly supported in $\omega\Subset\Omega$.

We now estimate the right-hand side of
\eqref{eq:ibp-local-fluid-coupling}. Since $\xi$ depends only on time and
$\chi$ only on space, we have
\begin{equation*}
\partial_t\left(\xi\chi^2\varphi_N\right)
=
\xi'\chi^2\varphi_N
+
\xi\chi^2\partial_t\varphi_N,
\end{equation*}
and
\begin{equation*}
\Delta\left(\xi\chi^2\varphi_N\right)
=
\xi\chi^2\Delta\varphi_N
+
\xi\varphi_N\Delta(\chi^2)
+
2\xi\nabla(\chi^2)\cdot\nabla\varphi_N.
\end{equation*}
Hence, from \eqref{eq:ibp-local-fluid-coupling},
\begin{equation}
\label{eq:ibp-local-fluid-coupling_inter}
\int_0^\tau\int_\Omega
\xi\chi^2|\varphi_N|^2\dx\dt
\leq
\sum_{i=1}^5J_i,
\end{equation}
where
\begin{equation*}
\begin{alignedat}{2}
J_1
&:=
\int_0^\tau\int_\omega
|\psi|\,|\xi'|\,\chi^2|\varphi_N|\dx\dt,
&\qquad
J_2
&:=
\int_0^\tau\int_\omega
|\psi|\,\xi\chi^2|\partial_t\varphi_N|\dx\dt,
\\
J_3
&:=
\kappa\int_0^\tau\int_\omega
|\psi|\,\xi\chi^2|\Delta\varphi_N|\dx\dt,
\end{alignedat}
\end{equation*}
and
\begin{equation*}
\begin{aligned}
J_4&:=\kappa\int_0^\tau\int_{\operatorname{supp}\Delta(\chi^2)}
|\psi|\,\xi|\varphi_N|\,|\Delta(\chi^2)|\dx\dt,\\
J_5&:=2\kappa\int_0^\tau\int_{\operatorname{supp}\nabla(\chi^2)}
|\psi|\,\xi|\nabla(\chi^2)|\,|\nabla\varphi_N|\dx\dt.
\end{aligned}
\end{equation*}

Let us estimate each $J_i$. From \eqref{eq:time-cutoff} and since $0\leq\xi\leq1$, we have
\begin{equation*}
|\xi'(t)|
\leq
\left(
\frac{|\xi'(t)|^2}{\xi(t)}
\right)^{1/2}
\xi(t)^{1/2}
\leq
\frac{C}{\tau}.
\end{equation*}
Therefore, by the Cauchy--Schwarz inequality in space and the definition of
the $\mathbb L^2$-norm,
\begin{equation*}
\begin{aligned}
J_1 \leq
\frac{C}{\tau}
\int_0^\tau
\|\psi(t)\|_{L^2(\omega)}
\|\varphi_N(t)\|_{L^2(\omega)}
\dt \leq
\frac{C}{\tau}
\int_0^\tau
\|\psi(t)\|_{L^2(\omega)}
\|\varphi(t)\|_{\mathbf H}
\dt.
\end{aligned}
\end{equation*}
Similarly, from the Cauchy--Schwarz inequality and
\eqref{eq:component-H2-estimate}, we have
\begin{equation*}
\begin{aligned}
J_2\leq
C
\int_0^\tau
\|\psi(t)\|_{L^2(\omega)}
\|\partial_t\varphi_N(t)\|_{L^2(\Omega)}
\dt
\leq
CM
\int_0^\tau
\|\psi(t)\|_{L^2(\omega)}
\|\varphi(t)\|_{\mathbf H}
\dt.
\end{aligned}
\end{equation*}

We now estimate the spatial terms. Again using
\eqref{eq:component-H2-estimate}, we obtain
\begin{equation*}
\begin{aligned}
J_3
\leq
C
\int_0^\tau
\|\psi(t)\|_{L^2(\omega)}
\|\Delta\varphi_N(t)\|_{L^2(\Omega)}
\dt
\leq
CM
\int_0^\tau
\|\psi(t)\|_{L^2(\omega)}
\|\varphi(t)\|_{\mathbf H}
\dt.
\end{aligned}
\end{equation*}
For the fourth and fifth terms, notice that, since $\chi\equiv1$ in
$\omega_0$ and $\chi$ is compactly supported in $\omega$, the supports of
$\nabla(\chi^2)$ and $\Delta(\chi^2)$ are contained in
$\omega\setminus\omega_0$. Therefore,
\begin{equation*}
\begin{aligned}
J_4
&\leq
C
\int_0^\tau
\|\psi(t)\|_{L^2(\operatorname{supp}\Delta(\chi^2))}
\|\varphi_N(t)\|_{L^2(\operatorname{supp}\Delta(\chi^2))}
\dt
\\
&\leq
C
\int_0^\tau
\|\psi(t)\|_{L^2(\omega)}
\|\varphi(t)\|_{\mathbf H}
\dt.
\end{aligned}
\end{equation*}
Similarly, using \eqref{eq:component-H2-estimate}, we have
\begin{equation*}
\begin{aligned}
J_5
&\leq
C
\int_0^\tau
\|\psi(t)\|_{L^2(\operatorname{supp}\nabla(\chi^2))}
\|\nabla\varphi_N(t)\|_{L^2(\operatorname{supp}\nabla(\chi^2))}
\dt
\\
&\leq
CM
\int_0^\tau
\|\psi(t)\|_{L^2(\omega)}
\|\varphi(t)\|_{\mathbf H}
\dt.
\end{aligned}
\end{equation*}

Putting together the estimates for $J_1,\ldots,J_5$ in
\eqref{eq:ibp-local-fluid-coupling_inter}, and using that $M\geq1$, we deduce
\begin{equation*}
\int_0^\tau\int_\Omega
\xi\chi^2|\varphi_N|^2\dx\dt
\leq
C\left(
M+\frac1\tau
\right)
\int_0^\tau
\|\psi(t)\|_{L^2(\omega)}
\|\varphi(t)\|_{\mathbf H}
\dt.
\end{equation*}
Lastly, the Cauchy--Schwarz inequality yields
\begin{equation*}
\begin{aligned}
\int_0^\tau\int_\Omega
\xi\chi^2|\varphi_N|^2\dx\dt
&\leq
C\left(
M+\frac1\tau
\right)
\left(
\int_0^\tau\int_\omega|\psi|^2\dx\dt
\right)^{1/2}\left(
\int_0^\tau\|\varphi(t)\|_{\mathbf H}^2\dt
\right)^{1/2}.
\end{aligned}
\end{equation*}
Combining this inequality with \eqref{eq:localized-left-phiN} gives
\eqref{eq:local-fluid-coupling}.
\end{proof}

\subsection{Combining the Stokes spectral inequality with the local estimate}

We now combine the Stokes spectral inequality with the local estimate obtained
in \Cref{lem:local-fluid-coupling}. Since
$\varphi(t)\in\mathbf E_M$ for every $t\in[0,\tau]$, we may apply
\Cref{prop:stokes-spectral-inequality} to $\varphi(t)$ on $\omega_0$, choosing
the missing component to be $m=N-1$. More precisely, we have
\begin{equation}
\label{eq:stokes-spectral-on-phi}
\|\varphi(t)\|_{\mathbf H}^2
\leq
C e^{C\sqrt M}
\int_{\omega_0}
\left(
\sum_{i=1}^{N-2}|\varphi_i(t,x)|^2
+
|\varphi_N(t,x)|^2
\right)\dx
\end{equation}
for every $t\in[0,\tau]$.

Integrating \eqref{eq:stokes-spectral-on-phi} over $I_\tau$ and using
\Cref{lem:local-fluid-coupling}, together with $\omega_0\subset\omega$, we
obtain
\begin{equation}
\label{eq:stokes-before-time-comparison}
\begin{aligned}
\int_{I_\tau}\|\varphi(t)\|_{\mathbf H}^2\dt
&\leq
C e^{C\sqrt M}
\int_0^\tau\int_\omega
\sum_{i=1}^{N-2}|\varphi_i|^2\dx\dt
\\
&\quad
+
C e^{C\sqrt M}
\left(
M+\frac1\tau
\right)
\left(
\int_0^\tau\int_\omega|\psi|^2\dx\dt
\right)^{1/2}
\left(
\int_0^\tau\|\varphi(t)\|_{\mathbf H}^2\dt
\right)^{1/2}.
\end{aligned}
\end{equation}
By \Cref{lem:stokes-time-comparison}, applied with $J=(0,\tau)$ and
$I=I_\tau$, and since $|I_\tau|=\tau/2$, we have
\begin{equation}
\label{eq:comparison-Itau}
\int_0^\tau\|\varphi(t)\|_{\mathbf H}^2\dt
\leq
2e^{2\nu M\tau}
\int_{I_\tau}\|\varphi(t)\|_{\mathbf H}^2\dt.
\end{equation}
Combining \eqref{eq:stokes-before-time-comparison} and
\eqref{eq:comparison-Itau}, and applying Young's inequality, we readily obtain the following result.

\begin{lem}
\label{lem:stokes-observation-intermediate}
Let $M\geq M_0$ and $\tau\in(0,1)$. There exists a constant $C>0$, depending
only on $\Omega$, $\omega_0$, $\omega$, $\nu$, and $\kappa$, but independent
of $M$, $\tau$, $\varphi$, and $\psi$, such that
\begin{equation}
\label{eq:stokes-observation-intermediate}
\begin{aligned}
\int_0^\tau\|\varphi(t)\|_{\mathbf H}^2\dt
&\leq
C e^{C\sqrt M+CM\tau}
\int_0^\tau\int_\omega
\sum_{i=1}^{N-2}|\varphi_i|^2\dx\dt
\\
&\quad
+
C e^{C\sqrt M+CM\tau}
\left(
M+\frac1\tau
\right)^2
\int_0^\tau\int_\omega|\psi|^2\dx\dt.
\end{aligned}
\end{equation}
\end{lem}

\subsection{The scalar component}

We now estimate the scalar component $\psi$ of the adjoint system
\eqref{eq:adjoint-system}. No spectral localization is imposed on the terminal
datum $\psi_\tau$. Instead, we use a standard observability inequality for the
heat equation with a source term.

\begin{lem}
\label{lem:nonhomogeneous-heat-observability}
Let $\tau\in(0,1)$, $F\in L^2(0,\tau;L^2(\Omega))$, and
$z_\tau\in L^2(\Omega)$. There exists a constant $C>0$, depending only on
$\Omega$, $\omega$, and $\kappa$, but independent of $\tau$, $F$, and
$z_\tau$, such that every solution of
\begin{equation}
\label{eq:nonhomogeneous-heat}
\begin{cases}
-z_t+\kappa A z=F
& t\in(0,\tau),\\
z(\tau)=z_\tau,
\end{cases}
\end{equation}
satisfies
\begin{equation}
\label{eq:nonhomogeneous-heat-observability}
\|z(0)\|_{L^2}^2
\leq
C e^{C/\tau}
\left(
\int_0^\tau\int_\omega|z|^2\dx\dt
+
\int_0^\tau\|F(t)\|_{L^2}^2\dt
\right).
\end{equation}
\end{lem}

Estimate \eqref{eq:nonhomogeneous-heat-observability} follows from the Carleman estimate for parabolic operators given in
\cite[Lemma~1.3]{FCG06}, combined with a standard energy estimate. The choice of the Carleman parameters and the corresponding bounds on the weights
give the factor $e^{C/\tau}$. We omit the details.

Applying \Cref{lem:nonhomogeneous-heat-observability} to the scalar equation
in \eqref{eq:adjoint-system}
and using
\begin{equation*}
\|\varphi_N(t)\|_{L^2}
\leq
\|\varphi(t)\|_{\mathbf H},
\qquad
t\in(0,\tau),
\end{equation*}
we obtain
\begin{equation}
\label{eq:scalar-component-estimate}
\|\psi(0)\|_{L^2}^2
\leq
C e^{C/\tau}
\left(
\int_0^\tau\int_\omega|\psi|^2\dx\dt
+
\int_0^\tau\|\varphi(t)\|_{\mathbf H}^2\dt
\right).
\end{equation}

\subsection{Proof of \Cref{thm:mixed-observability}}

We are now in position to prove the main result of this section. Let
$M_\ast\geq M_0$ be large enough so that $M_\ast>1$. Then, for every
$M\geq M_\ast$, the choice $\tau=1/\sqrt M$ satisfies $\tau\in(0,1)$.

Set
\begin{equation*}
\mathcal O_\tau
:=
\int_0^\tau\int_\omega
\left(
\sum_{i=1}^{N-2}|\varphi_i|^2+|\psi|^2
\right)\dx\dt.
\end{equation*}
We first estimate the Stokes component. From
\Cref{lem:stokes-observation-intermediate}, we have
\begin{equation}
\label{eq:est_phi_full_time}
\begin{aligned}
\int_0^\tau\|\varphi(t)\|_{\mathbf H}^2\dt
&\leq
C e^{C\sqrt M+CM\tau}
\int_0^\tau\int_\omega
\sum_{i=1}^{N-2}|\varphi_i|^2\dx\dt
\\
&\quad
+
C e^{C\sqrt M+CM\tau}
\left(
M+\frac1\tau
\right)^2
\int_0^\tau\int_\omega|\psi|^2\dx\dt.
\end{aligned}
\end{equation}
Since $\tau=1/\sqrt M$, we have
\begin{equation*}
M\tau=\sqrt M,
\qquad
\frac1\tau=\sqrt M,
\qquad
M+\frac1\tau\leq2M
\end{equation*}
for every $M\geq1$. We also recall that, for every $q>0$ and
$\varepsilon>0$,
\begin{equation}
\label{eq:est_unif_prod_poly_exp_mixed}
\sup_{M\geq1}
M^q e^{-\varepsilon\sqrt M}
<+\infty.
\end{equation}
Thus, the factor $M^2$ in \eqref{eq:est_phi_full_time} can be absorbed into
the exponential by increasing its constant. We conclude that
\begin{equation}
\label{eq:final-stokes-integral-mixed}
\int_0^\tau\|\varphi(t)\|_{\mathbf H}^2\dt
\leq
C e^{C\sqrt M}\mathcal O_\tau.
\end{equation}

Using item \ref{eq:phi-energy-decay} of
\Cref{lem:stokes-spectral-regularity}, we also have
\begin{equation*}
\|\varphi(0)\|_{\mathbf H}^2
\leq
\frac1\tau
\int_0^\tau\|\varphi(t)\|_{\mathbf H}^2\dt.
\end{equation*}
Since $\tau^{-1}=\sqrt M$, we use
\eqref{eq:est_unif_prod_poly_exp_mixed} once more. Then, it follows from
\eqref{eq:final-stokes-integral-mixed} that
\begin{equation}
\label{eq:final-phi0-estimate-mixed}
\|\varphi(0)\|_{\mathbf H}^2
\leq
C e^{C\sqrt M}\mathcal O_\tau.
\end{equation}

We now estimate the scalar component. From
\eqref{eq:scalar-component-estimate} and $\tau^{-1}=\sqrt M$, we have
\begin{equation*}
\|\psi(0)\|_{L^2}^2
\leq
C e^{C\sqrt M}
\left(
\mathcal O_\tau
+
\int_0^\tau\|\varphi(t)\|_{\mathbf H}^2\dt
\right).
\end{equation*}
Using \eqref{eq:final-stokes-integral-mixed} and increasing $C$ if necessary,
we obtain
\begin{equation}
\label{eq:final-psi0-estimate-mixed}
\|\psi(0)\|_{L^2}^2
\leq
C e^{C\sqrt M}\mathcal O_\tau.
\end{equation}
Finally, adding \eqref{eq:final-phi0-estimate-mixed} and
\eqref{eq:final-psi0-estimate-mixed} gives
\eqref{eq:mixed-observability}.

\section{Lebeau--Robbiano iteration and null controllability}
\label{sec:linear-control}

In this section, we use \Cref{thm:mixed-observability} to prove the null
controllability of the linearized system. We consider a vector control
$h=(h_1,\ldots,h_N)\in L^2(0,T;L^2(\omega)^N)$, with
$h_{N-1}\equiv h_N\equiv0$, and a scalar control
$h_\theta\in L^2(0,T;L^2(\omega))$. The corresponding controlled system is
\begin{equation}
\label{eq:controlled-linear-system}
\begin{cases}
u_t+\nu\mathbf A u
=
\Pi(\theta e_N)+\Pi(\mathbf 1_\omega h)
& t\in(0,T),\\
\theta_t+\kappa A\theta
=
\mathbf 1_\omega h_\theta
& t\in(0,T),\\
(u(0),\theta(0))=(u_0,\theta_0).
\end{cases}
\end{equation}
In dimension two, the vector control vanishes identically, while in dimension
three only its first component may be nonzero.

We can now state the main result of this section.

\begin{theo}
\label{thm:quantitative-linear-controllability}
There exist $T_0\in(0,1)$ and $C>0$ such that, for every
$T\in(0,T_0)$ and every initial datum
$(u_0,\theta_0)\in\mathbf H\times L^2(\Omega)$, there exist controls
\begin{equation*}
h\in L^2(0,T;L^2(\omega)^N),
\qquad
h_\theta\in L^2(0,T;L^2(\omega)),
\end{equation*}
with $h_{N-1}\equiv h_N\equiv0$, such that the corresponding solution of
\eqref{eq:controlled-linear-system} satisfies
\begin{equation*}
u(T)=0,
\qquad
\theta(T)=0.
\end{equation*}
Moreover,
\begin{equation}
\label{eq:quantitative-control-cost}
\begin{aligned}
\|h\|_{L^2(0,T;L^2(\omega)^N)}
+
\|h_\theta\|_{L^2(0,T;L^2(\omega))} \leq
C\exp\left(\frac{C}{T}\right)
\left(
\|u_0\|_{\mathbf H}
+
\|\theta_0\|_{L^2}
\right).
\end{aligned}
\end{equation}
\end{theo}

The proof of \Cref{thm:quantitative-linear-controllability} is developed in
the following subsections. The construction alternates control intervals
with intervals of free evolution. On each control interval, the controls are
chosen so that
\begin{equation*}
\mathbf P_Mu=0,
\qquad
\theta=0
\end{equation*}
at the end of the interval. During the subsequent free evolution, the
temperature remains zero and the high Stokes frequencies decay.

\subsection{The two basic steps}

For $M>0$, we set
\begin{equation}
\label{eq:tau-sigma-LR}
\tau=\frac1{\sqrt M},
\qquad
\sigma=\frac{L_s}{\sqrt M},
\end{equation}
where $L_s>0$ will be chosen below. Increasing $M_\ast$ if necessary, we
assume that $\tau\in(0,1)$ for every $M\geq M_\ast$.

We first construct the controls on the interval $(0,\tau)$.

\begin{lem}
\label{lem:thermal-loop-step}
There exist constants $C_1>0$ and $M_\ast\geq1$ such that, for every
$M\geq M_\ast$ and every
$(u_0,\theta_0)\in\mathbf H\times L^2(\Omega)$, there exist controls
\begin{equation*}
h\in L^2(0,\tau;L^2(\omega)^N),
\qquad
h_\theta\in L^2(0,\tau;L^2(\omega)),
\end{equation*}
with $h_{N-1}\equiv h_N\equiv0$, such that the solution of
\eqref{eq:controlled-linear-system} satisfies
\begin{equation}
\label{eq:thermal-loop-final-constraints}
\mathbf P_Mu(\tau)=0,
\qquad
\theta(\tau)=0.
\end{equation}
Moreover,
\begin{equation}
\label{eq:thermal-loop-cost}
\begin{aligned}
&\|h\|_{L^2(0,\tau;L^2(\omega)^N)}
+
\|h_\theta\|_{L^2(0,\tau;L^2(\omega))}
\\
&\qquad
\leq
C_1e^{C_1\sqrt M}
\left(
\|u_0\|_{\mathbf H}
+
\|\theta_0\|_{L^2}
\right),
\end{aligned}
\end{equation}
and
\begin{equation}
\label{eq:thermal-loop-velocity-rough}
\|u(\tau)\|_{\mathbf H}
\leq
C_1e^{C_1\sqrt M}
\left(
\|u_0\|_{\mathbf H}
+
\|\theta_0\|_{L^2}
\right).
\end{equation}
\end{lem}

\begin{proof}
Let $\Lambda$ be the linear subspace of
$L^2\bigl(0,\tau;L^2(\omega)^{N-2}\times L^2(\omega)\bigr)$ given by
\begin{equation*}
\Lambda
:=
\left\{
\left.
\left(
\varphi_1,\ldots,\varphi_{N-2},\psi
\right)\right|_{(0,\tau)\times\omega}
:
\begin{array}{l}
(\varphi,\psi)\text{ solves }\eqref{eq:adjoint-system},\\
(\varphi_\tau,\psi_\tau)
\in\mathbf E_M\times L^2(\Omega)
\end{array}
\right\}.
\end{equation*}
When $N=2$, only the scalar component $\psi$ appears in the definition of
$\Lambda$.

On $\Lambda$, we define
\begin{equation}
\label{eq:def-functional-thermal-loop}
\mathcal L
\left(
\varphi_1,\ldots,\varphi_{N-2},\psi
\right)
:=
-\left\langle u_0,\varphi(0)\right\rangle_{\mathbf H}
-\left\langle\theta_0,\psi(0)\right\rangle_{L^2}.
\end{equation}
By \Cref{thm:mixed-observability}, this functional is well defined and
satisfies
\begin{equation}
\label{eq:functional-bound-thermal-loop}
\begin{aligned}
\left|
\mathcal L
\left(
\varphi_1,\ldots,\varphi_{N-2},\psi
\right)
\right|
&\leq
Ce^{C\sqrt M}
\left(
\|u_0\|_{\mathbf H}
+
\|\theta_0\|_{L^2}
\right)
\\
&\quad\times
\left(
\int_0^\tau\int_\omega
\left(
\sum_{i=1}^{N-2}|\varphi_i|^2+|\psi|^2
\right)\dx\dt
\right)^{1/2}.
\end{aligned}
\end{equation}
The Hahn--Banach theorem allows us to extend $\mathcal L$ to
$L^2\bigl(0,\tau;L^2(\omega)^{N-2}\times L^2(\omega)\bigr)$ while preserving
this bound. By the Riesz representation theorem, there exist
$h_1,\ldots,h_{N-2}$ and $h_\theta$ such that
\begin{equation}
\label{eq:thermal-loop-duality-choice}
\int_0^\tau\int_\omega
\left(
\sum_{i=1}^{N-2}h_i\varphi_i+h_\theta\psi
\right)\dx\dt
=
-\left\langle u_0,\varphi(0)\right\rangle_{\mathbf H}
-\left\langle\theta_0,\psi(0)\right\rangle_{L^2}
\end{equation}
for every adjoint solution with terminal datum in
$\mathbf E_M\times L^2(\Omega)$. Setting $h_{N-1}=h_N=0$, we obtain a vector
control $h=(h_1,\ldots,h_N)$ satisfying \eqref{eq:thermal-loop-cost}. In
dimension two, this simply gives $h=(0,0)$.

Let $(u,\theta)$ be the corresponding solution of
\eqref{eq:controlled-linear-system}. By duality between
\eqref{eq:controlled-linear-system} and \eqref{eq:adjoint-system}, we have
\begin{equation*}
\begin{aligned}
&\left\langle u(\tau),\varphi_\tau\right\rangle_{\mathbf H}
+
\left\langle\theta(\tau),\psi_\tau\right\rangle_{L^2}
\\
&\quad
=
\left\langle u_0,\varphi(0)\right\rangle_{\mathbf H}
+
\left\langle\theta_0,\psi(0)\right\rangle_{L^2}
+
\int_0^\tau\int_\omega
\left(
\sum_{i=1}^{N-2}h_i\varphi_i+h_\theta\psi
\right)\dx\dt.
\end{aligned}
\end{equation*}
It follows from \eqref{eq:thermal-loop-duality-choice} that
\begin{equation*}
\left\langle u(\tau),\varphi_\tau\right\rangle_{\mathbf H}
+
\left\langle\theta(\tau),\psi_\tau\right\rangle_{L^2}
=0
\end{equation*}
for every
$(\varphi_\tau,\psi_\tau)\in\mathbf E_M\times L^2(\Omega)$. Taking first
$\psi_\tau=0$ and then $\varphi_\tau=0$, we obtain
\eqref{eq:thermal-loop-final-constraints}.

Finally, item \textup{(i)} of \Cref{thm:linear-wellposedness} and
\eqref{eq:thermal-loop-cost} give
\begin{equation*}
\begin{aligned}
\|u(\tau)\|_{\mathbf H}
&\leq
C\left(
\|u_0\|_{\mathbf H}
+
\|\theta_0\|_{L^2}
+
\|h\|_{L^2(0,\tau;L^2(\omega)^N)}
+
\|h_\theta\|_{L^2(0,\tau;L^2(\omega))}
\right)
\\
&\leq
C_1e^{C_1\sqrt M}
\left(
\|u_0\|_{\mathbf H}
+
\|\theta_0\|_{L^2}
\right),
\end{aligned}
\end{equation*}
which proves \eqref{eq:thermal-loop-velocity-rough}.
\end{proof}

We now set the controls equal to zero on $(\tau,\tau+\sigma)$. By
\eqref{eq:thermal-loop-final-constraints}, the temperature remains zero on
this interval and $u(\tau)$ contains only Stokes frequencies larger than
$M$. We use these two properties in the following lemma.

\begin{lem}
\label{lem:free-stokes-dissipation-step}
Let $L_s>C_1/\nu$, where $C_1$ is the constant in
\Cref{lem:thermal-loop-step}. Then, for every $M\geq M_\ast$, the solution
constructed in \Cref{lem:thermal-loop-step} satisfies
\begin{equation}
\label{eq:one-step-reduction}
\|u(\tau+\sigma)\|_{\mathbf H}
+
\|\theta(\tau+\sigma)\|_{L^2}
\leq
C_1e^{-\gamma\sqrt M}
\left(
\|u_0\|_{\mathbf H}
+
\|\theta_0\|_{L^2}
\right),
\end{equation}
where $\gamma:=\nu L_s-C_1>0$.
\end{lem}

\begin{proof}
Since the controls vanish on $(\tau,\tau+\sigma)$ and $\theta(\tau)=0$, the
uniqueness of the solution to the homogeneous heat equation gives
\begin{equation}
\label{eq:theta-zero-after-loop}
\theta(\tau+r)=0
\qquad
\text{for every }r\in[0,\sigma].
\end{equation}
Consequently,
\begin{equation*}
u(\tau+\sigma)
=
e^{-\nu\sigma\mathbf A}u(\tau).
\end{equation*}
Since $\mathbf P_Mu(\tau)=0$, the spectral decomposition of the Stokes
semigroup gives
\begin{equation*}
\|u(\tau+\sigma)\|_{\mathbf H}
\leq
e^{-\nu M\sigma}\|u(\tau)\|_{\mathbf H}.
\end{equation*}
Using \eqref{eq:thermal-loop-velocity-rough} and
$\sigma=L_s/\sqrt M$, we obtain
\begin{equation*}
\begin{aligned}
\|u(\tau+\sigma)\|_{\mathbf H}
&\leq
C_1e^{-\nu M\sigma+C_1\sqrt M}
\left(
\|u_0\|_{\mathbf H}
+
\|\theta_0\|_{L^2}
\right)
\\
&=
C_1e^{-(\nu L_s-C_1)\sqrt M}
\left(
\|u_0\|_{\mathbf H}
+
\|\theta_0\|_{L^2}
\right).
\end{aligned}
\end{equation*}
Together with \eqref{eq:theta-zero-after-loop}, this proves
\eqref{eq:one-step-reduction}.
\end{proof}

From now on, $L_s>C_1/\nu$ is fixed independently of $M$.

\subsection{The iterative construction}

We now iterate the two steps obtained in the previous subsection. Each
interval is divided into a control interval, where
\Cref{lem:thermal-loop-step} is applied, followed by an interval of free
evolution, where \Cref{lem:free-stokes-dissipation-step} is used.

Let $q>1$ and $M_0\geq M_\ast$ be parameters to be chosen below. For every
$j\in\mathbb N$, we set
\begin{equation*}
M_j=M_0q^j,
\qquad
\tau_j=\frac1{\sqrt{M_j}},
\qquad
\sigma_j=\frac{L_s}{\sqrt{M_j}},
\qquad
\ell_j=\tau_j+\sigma_j,
\end{equation*}
where $L_s$ is the constant fixed at the end of the previous subsection. We
also define
\begin{equation*}
t_0=0,
\qquad
t_{j+1}=t_j+\ell_j,
\qquad
j\in\mathbb N.
\end{equation*}
Thus, $t_j=\sum_{m=0}^{j-1}\ell_m$ for every $j\geq1$. Since
$M_j=M_0q^j$, the limiting time of the construction is
\begin{equation}
\label{eq:Tstar-definition-LR}
\begin{aligned}
T_\ast
:=
\lim_{n\to\infty}t_n
&=
\sum_{j=0}^\infty(\tau_j+\sigma_j)
\\
&=
\frac{1+L_s}{\sqrt{M_0}}
\sum_{j=0}^\infty q^{-j/2}
=
\frac{1+L_s}{\sqrt{M_0}(1-q^{-1/2})}
<+\infty.
\end{aligned}
\end{equation}

We fix $q>1$ sufficiently close to $1$ so that
\begin{equation}
\label{eq:q-choice-LR}
\frac{\gamma}{q^{1/2}-1}>2C_1,
\end{equation}
where $C_1$ is the constant in \eqref{eq:thermal-loop-cost} and $\gamma$ is
the constant in \eqref{eq:one-step-reduction}. Once $q$ is fixed, we increase
$M_\ast$ if necessary so that
\begin{equation*}
\frac{1+L_s}{\sqrt{M_\ast}(1-q^{-1/2})}<1,
\end{equation*}
and then take $M_0\geq M_\ast$. It follows from
\eqref{eq:Tstar-definition-LR} that $T_\ast<1$. Moreover,
$M_j\geq M_\ast$ and $\tau_j\in(0,1)$ for every $j\in\mathbb N$, so the two
steps of the previous subsection can be applied at every level.

We construct the controls on $(0,T_\ast)$. For every $j\in\mathbb N$, the
interval $(t_j,t_{j+1})$ is divided into
\begin{equation*}
(t_j,t_j+\tau_j)
\qquad\text{and}\qquad
(t_j+\tau_j,t_{j+1}).
\end{equation*}
On the first subinterval, we apply \Cref{lem:thermal-loop-step}, translated
to the initial time $t_j$, with $M=M_j$ and initial datum
$(u(t_j),\theta(t_j))$. This gives controls
\begin{equation*}
h^j\in L^2(t_j,t_j+\tau_j;L^2(\omega)^N),
\qquad
h_\theta^j\in L^2(t_j,t_j+\tau_j;L^2(\omega)),
\end{equation*}
with $h_{N-1}^j\equiv h_N^j\equiv0$, such that
\begin{equation*}
\mathbf P_{M_j}u(t_j+\tau_j)=0,
\qquad
\theta(t_j+\tau_j)=0.
\end{equation*}
On the second subinterval, $(t_j+\tau_j,t_{j+1})$, both controls are set
equal to zero. Applying \Cref{lem:free-stokes-dissipation-step}, translated
to the initial time $t_j+\tau_j$, we obtain
\begin{equation*}
\begin{aligned}
&\|u(t_{j+1})\|_{\mathbf H}
+
\|\theta(t_{j+1})\|_{L^2}
\\
&\qquad
\leq
C e^{-\gamma\sqrt{M_j}}
\left(
\|u(t_j)\|_{\mathbf H}
+
\|\theta(t_j)\|_{L^2}
\right).
\end{aligned}
\end{equation*}

Set
\begin{equation*}
A_j
:=
\|u(t_j)\|_{\mathbf H}
+
\|\theta(t_j)\|_{L^2},
\qquad
j\in\mathbb N.
\end{equation*}
The previous estimate becomes
\begin{equation*}
A_{j+1}
\leq
C e^{-\gamma\sqrt{M_j}}A_j,
\qquad
j\in\mathbb N.
\end{equation*}
Iterating it, we obtain, for every $n\geq1$,
\begin{equation}
\label{eq:Aj-iterated-reduction-LR}
A_n
\leq
C^n
\exp\left(
-\gamma\sum_{j=0}^{n-1}\sqrt{M_j}
\right)
A_0.
\end{equation}
Since $M_j=M_0q^j$, we have
\begin{equation*}
\sum_{j=0}^{n-1}\sqrt{M_j}
=
\sqrt{M_0}
\sum_{j=0}^{n-1}q^{j/2}
=
\sqrt{M_0}
\frac{q^{n/2}-1}{q^{1/2}-1}.
\end{equation*}
In particular,
\begin{equation*}
n\log C
-
\gamma\sum_{j=0}^{n-1}\sqrt{M_j}
\longrightarrow
-\infty
\qquad
\text{as }n\to\infty.
\end{equation*}
It follows from \eqref{eq:Aj-iterated-reduction-LR} that
\begin{equation}
\label{eq:Aj-goes-to-zero-LR}
A_n\longrightarrow0
\qquad
\text{as }n\to\infty.
\end{equation}

It remains to check that the controls obtained by concatenation belong to the
required spaces. We extend $(h^j,h_\theta^j)$ by zero to the whole interval
$(t_j,t_{j+1})$ and define
\begin{equation*}
h(t)=h^j(t),
\qquad
h_\theta(t)=h_\theta^j(t),
\qquad
t\in(t_j,t_{j+1}).
\end{equation*}
Since the intervals $(t_j,t_{j+1})$ are pairwise disjoint, we have
\begin{equation*}
\begin{aligned}
&\|h\|_{L^2(0,T_\ast;L^2(\omega)^N)}^2
+
\|h_\theta\|_{L^2(0,T_\ast;L^2(\omega))}^2
\\
&\qquad
=
\sum_{j=0}^\infty
\left(
\|h^j\|_{L^2(t_j,t_{j+1};L^2(\omega)^N)}^2
+
\|h_\theta^j\|_{L^2(t_j,t_{j+1};L^2(\omega))}^2
\right).
\end{aligned}
\end{equation*}
Thus, it is enough to prove that the series on the right-hand side is finite.
In fact, we prove the stronger estimate
\begin{equation}
\label{eq:global-control-cost-M0}
\begin{aligned}
\sum_{j=0}^\infty
\bigl(
&\|h^j\|_{L^2(t_j,t_{j+1};L^2(\omega)^N)}
\\
&+
\|h_\theta^j\|_{L^2(t_j,t_{j+1};L^2(\omega))}
\bigr)
\leq
Ce^{C\sqrt{M_0}}A_0.
\end{aligned}
\end{equation}

From the control estimate in \Cref{lem:thermal-loop-step}, applied at the
$j$-th level, we have
\begin{equation*}
\begin{aligned}
&\|h^j\|_{L^2(t_j,t_{j+1};L^2(\omega)^N)}
+
\|h_\theta^j\|_{L^2(t_j,t_{j+1};L^2(\omega))}
\\
&\qquad
\leq
C_1e^{C_1\sqrt{M_j}}A_j.
\end{aligned}
\end{equation*}
In particular,
\begin{equation}
\label{eq:first-block-control-cost-LR}
\begin{aligned}
&\|h^0\|_{L^2(t_0,t_1;L^2(\omega)^N)}
+
\|h_\theta^0\|_{L^2(t_0,t_1;L^2(\omega))}
\\
&\qquad
\leq
C_1e^{C_1\sqrt{M_0}}A_0.
\end{aligned}
\end{equation}
Using \eqref{eq:Aj-iterated-reduction-LR}, for every $j\geq1$ we obtain
\begin{equation}
\label{eq:block-control-cost-iterated-LR}
\begin{aligned}
&\|h^j\|_{L^2(t_j,t_{j+1};L^2(\omega)^N)}
+
\|h_\theta^j\|_{L^2(t_j,t_{j+1};L^2(\omega))}
\\
&\qquad
\leq
C_1C^j
\exp\left(
C_1\sqrt{M_j}
-
\gamma\sum_{m=0}^{j-1}\sqrt{M_m}
\right)
A_0.
\end{aligned}
\end{equation}

We now estimate the exponential factor in
\eqref{eq:block-control-cost-iterated-LR}. Since
$\sqrt{M_j}=\sqrt{M_0}q^{j/2}$ and
\begin{equation*}
\sum_{m=0}^{j-1}\sqrt{M_m}
=
\sqrt{M_0}
\frac{q^{j/2}-1}{q^{1/2}-1},
\end{equation*}
we have
\begin{equation*}
\begin{aligned}
C_1\sqrt{M_j}
-
\gamma\sum_{m=0}^{j-1}\sqrt{M_m}
&=
C_1\sqrt{M_0}q^{j/2}
-
\gamma\sqrt{M_0}
\frac{q^{j/2}-1}{q^{1/2}-1}
\\
&=
\sqrt{M_0}
\left[
\left(
C_1-\frac{\gamma}{q^{1/2}-1}
\right)
q^{j/2}
+
\frac{\gamma}{q^{1/2}-1}
\right].
\end{aligned}
\end{equation*}
By \eqref{eq:q-choice-LR},
\begin{equation*}
C_1-\frac{\gamma}{q^{1/2}-1}
<
-C_1,
\end{equation*}
and consequently
\begin{equation}
\label{eq:exponent-control-cost-M0}
C_1\sqrt{M_j}
-
\gamma\sum_{m=0}^{j-1}\sqrt{M_m}
\leq
-C_1\sqrt{M_0}q^{j/2}
+
C\sqrt{M_0},
\end{equation}
where $C>0$ is independent of $j$ and $M_0$. Inserting
\eqref{eq:exponent-control-cost-M0} into
\eqref{eq:block-control-cost-iterated-LR}, we obtain, for every $j\geq1$,
\begin{equation}
\label{eq:block-control-cost-summable-M0}
\begin{aligned}
&\|h^j\|_{L^2(t_j,t_{j+1};L^2(\omega)^N)}
+
\|h_\theta^j\|_{L^2(t_j,t_{j+1};L^2(\omega))}
\\
&\qquad
\leq
Ce^{C\sqrt{M_0}}
C^j e^{-C_1\sqrt{M_0}q^{j/2}}A_0.
\end{aligned}
\end{equation}
Since $M_0\geq1$ and $q>1$ is fixed, the series
\begin{equation*}
\sum_{j=1}^\infty
C^j e^{-C_1\sqrt{M_0}q^{j/2}}
\end{equation*}
is bounded uniformly with respect to $M_0$. Combining this fact with
\eqref{eq:first-block-control-cost-LR} and
\eqref{eq:block-control-cost-summable-M0}, we obtain
\eqref{eq:global-control-cost-M0}. In particular,
\begin{equation*}
h\in L^2(0,T_\ast;L^2(\omega)^N),
\qquad
h_\theta\in L^2(0,T_\ast;L^2(\omega)),
\end{equation*}
with $h_{N-1}\equiv h_N\equiv0$, and
\begin{equation}
\label{eq:global-control-L2-cost-LR}
\begin{aligned}
&\|h\|_{L^2(0,T_\ast;L^2(\omega)^N)}
+
\|h_\theta\|_{L^2(0,T_\ast;L^2(\omega))}
\\
&\qquad
\leq
Ce^{C\sqrt{M_0}}A_0.
\end{aligned}
\end{equation}

We now pass to the limit in the state. By
\Cref{thm:linear-wellposedness}, the control obtained by concatenation gives
a solution satisfying
\begin{equation*}
(u,\theta)
\in
C([0,T_\ast];\mathbf H\times L^2(\Omega)).
\end{equation*}
By uniqueness, this solution agrees on each interval $(0,t_n)$ with the
solution obtained by concatenating the first $n$ steps. Hence, by
\eqref{eq:Aj-goes-to-zero-LR},
\begin{equation*}
\|u(t_j)\|_{\mathbf H}
+
\|\theta(t_j)\|_{L^2}
\longrightarrow0
\qquad
\text{as }j\to\infty.
\end{equation*}
Since $t_j\uparrow T_\ast$, the continuity of the solution gives
\begin{equation*}
(u(t_j),\theta(t_j))
\longrightarrow
(u(T_\ast),\theta(T_\ast))
\qquad
\text{in }\mathbf H\times L^2(\Omega).
\end{equation*}
Therefore,
\begin{equation}
\label{eq:null-at-Tstar-LR}
u(T_\ast)=0,
\qquad
\theta(T_\ast)=0.
\end{equation}

\begin{rmk}
The temperature is zero at the end of every control interval:
\begin{equation*}
\theta(t_j+\tau_j)=0,
\qquad
j\in\mathbb N.
\end{equation*}
Since the scalar equation is homogeneous during the following interval of
free evolution, we also have
\begin{equation*}
\theta(t)=0
\qquad
\text{for every }t\in[t_j+\tau_j,t_{j+1}].
\end{equation*}
In particular, $\theta(t_j)=0$ and
$A_j=\|u(t_j)\|_{\mathbf H}$ for every $j\geq1$. We have kept the definition
of $A_j$ involving both components since the initial datum
$(u_0,\theta_0)$ is arbitrary.
\end{rmk}

\subsection{The control cost in small time}

We now conclude the proof by expressing the limiting time and the control
estimate obtained in the iterative construction in terms of the prescribed
time $T$.

\begin{proof}[Proof of \Cref{thm:quantitative-linear-controllability}]
Set $A_0:=\|u_0\|_{\mathbf H}+\|\theta_0\|_{L^2}$. From
\eqref{eq:global-control-L2-cost-LR}, the controls constructed in the previous section satisfy
\begin{equation}
\label{eq:global-cost-M0-theorem}
\|h\|_{L^2(0,T_\ast;L^2(\omega)^N)}
+
\|h_\theta\|_{L^2(0,T_\ast;L^2(\omega))}
\leq
Ce^{C\sqrt{M_0}}A_0.
\end{equation}

On the other hand, once $q>1$ has been fixed,
\eqref{eq:Tstar-definition-LR} gives
\begin{equation*}
T_\ast(M_0)=K M_0^{-1/2},
\qquad
K
:=
(1+L_s)\sum_{j=0}^\infty q^{-j/2}
=
\frac{1+L_s}{1-q^{-1/2}}
>0.
\end{equation*}
In particular, $K$ is independent of $M_0$, and the function
$M_0\mapsto T_\ast(M_0)$ is continuous, strictly decreasing, and converges
to zero as $M_0\to+\infty$.

Let $M_\ast$ be the threshold fixed in the iterative construction and set
\begin{equation*}
T_0
:=
\min\left\{
\frac12,T_\ast(M_\ast)
\right\}.
\end{equation*}
For every $T\in(0,T_0)$, choose $M_0=K^2/T^2$. Since
$T<T_\ast(M_\ast)=K/\sqrt{M_\ast}$, we have $M_0>M_\ast$. Moreover,
$T_\ast(M_0)=T$ and $\sqrt{M_0}=K/T$. Therefore,
\eqref{eq:null-at-Tstar-LR} gives $u(T)=0$ and $\theta(T)=0$.

Finally, since $K$ is fixed, \eqref{eq:global-cost-M0-theorem} yields
\begin{equation*}
\begin{aligned}
\|h\|_{L^2(0,T;L^2(\omega)^N)}
+
\|h_\theta\|_{L^2(0,T;L^2(\omega))}
\leq
C\exp\left(\frac{C}{T}\right)
\left(
\|u_0\|_{\mathbf H}
+
\|\theta_0\|_{L^2}
\right).
\end{aligned}
\end{equation*}
This proves \eqref{eq:quantitative-control-cost}.
\end{proof}

\begin{rmk}
\label{rem:any-final-time}
The restriction $T\in(0,T_0)$ in
\Cref{thm:quantitative-linear-controllability} is used to describe the
small-time control cost. The null controllability result itself holds for
every $T>0$. Indeed, let
$
\widehat T
:=
\min\left\{
\frac{T}{2},\frac{T_0}{2}
\right\}.
$
Applying \Cref{thm:quantitative-linear-controllability} on
$(0,\widehat T)$ and extending the controls by zero on
$(\widehat T,T)$, we obtain a solution that vanishes at time $\widehat T$.
By linearity and uniqueness for the homogeneous system, it remains identically zero on $[\widehat T,T]$ and, in particular,
$u(T)=\theta(T)=0$.
\end{rmk}

\section{Proof of the nonlinear controllability result}
\label{sec:nonlinear}

In this section, we prove
\Cref{thm:main-nonlinear-controllability}. The two dimensions are treated
separately. We first consider the three-dimensional case, which requires more
attention because the nonlinear system is only locally well posed in the
space used in the iteration. We then consider the two-dimensional case. There,
the global existence and uniqueness of weak solutions allow us to apply the
same argument in the energy space $\mathbb X^0$, leading to a slightly
stronger result.

We use the time-iteration argument introduced in \cite{Mar26}. This method
transfers both null controllability and the corresponding control cost from
the linearized system to the nonlinear one by means of standard energy and
interpolation estimates. In particular, it avoids the construction of the
more elaborate weighted spaces used in the source term method of
\cite{LTT13}.

The argument in \cite{Mar26} is presented for globally well-posed equations.
It is also observed there that the method can be applied to locally
well-posed systems, provided that the existence of the nonlinear solution is
ensured at every step of the iteration, although the details of this extension
are not developed. We provide these details for the three-dimensional
Boussinesq system. We include the argument for completeness and the reader's
convenience, making no claim of novelty concerning the time-iteration
mechanism or its extension to locally well-posed equations. The point that
requires some care is to verify that the smallness conditions needed for the
existence of the nonlinear solution are preserved throughout the iteration.

We begin with the three-dimensional case.

\subsection{The nonlinear and linearized solution maps in dimension three}

Let $T\in(0,1)$, $y_0=(u_0,\theta_0)\in\mathbb X^1$, and
\begin{equation*}
h=(h_1,h_2)\in L^2(0,T;L^2(\omega)\times L^2(\omega)).
\end{equation*}
Here $h_1$ acts on the first component of the velocity equation, while $h_2$
acts on the temperature equation.

We denote by $S_t(y_0,h)$, $t\in[0,T]$, the solution map associated with the
nonlinear system
\begin{equation}
\label{eq:nonlinear-boussinesq-3d}
\begin{cases}
u_t+\nu\mathbf A u+\mathbf B(u,u)
=
\Pi(\theta e_3)+\Pi(\chi_\omega h_1e_1)
& t\in(0,T),\\
\theta_t+\kappa A\theta+u\cdot\nabla\theta
=
\chi_\omega h_2
& t\in(0,T),\\
(u(0),\theta(0))=(u_0,\theta_0).
\end{cases}
\end{equation}
More precisely, whenever the solution to
\eqref{eq:nonlinear-boussinesq-3d} is well defined on $[0,T]$, we set
\begin{equation*}
S_t(y_0,h):=(u(t),\theta(t)),
\qquad
t\in[0,T].
\end{equation*}

Similarly, we denote by $\Sigma_t(y_0,h)$, $t\in[0,T]$, the solution map
associated with the linearized system
\begin{equation}
\label{eq:linear-boussinesq-3d}
\begin{cases}
z_t+\nu\mathbf A z
=
\Pi(\xi e_3)+\Pi(\chi_\omega h_1e_1)
& t\in(0,T),\\
\xi_t+\kappa A\xi
=
\chi_\omega h_2
& t\in(0,T),\\
(z(0),\xi(0))=(u_0,\theta_0).
\end{cases}
\end{equation}
That is,
\begin{equation*}
\Sigma_t(y_0,h):=(z(t),\xi(t)),
\qquad
t\in[0,T].
\end{equation*}

We first give a consequence of
\Cref{thm:quantitative-linear-controllability} for initial data in
$\mathbb X^1$.

\begin{cor}
\label{cor:quantitative-control-cost-strong}
Assume that $N=3$. There exist $T_0\in(0,1)$ and $C>0$ such that, for every
$T\in(0,T_0)$ and every $y_0\in\mathbb X^1$, there exists a control
\begin{equation*}
h=(h_1,h_2)\in L^2(0,T;L^2(\omega)\times L^2(\omega))
\end{equation*}
such that
\begin{equation*}
\Sigma_T(y_0,h)=0
\end{equation*}
and
\begin{equation}
\label{eq:quantitative-control-cost-strong}
\|h\|_{L^2(0,T;L^2(\omega)\times L^2(\omega))}
\leq
C\exp\left(\frac{C}{T}\right)\|y_0\|_{\mathbb X^1}.
\end{equation}
Moreover, the corresponding solution $(z,\xi)$ satisfies
\begin{equation*}
(z,\xi)\in
C([0,T];\mathbb X^1)\cap L^2(0,T;\mathbb X^2).
\end{equation*}
\end{cor}

\begin{proof}
Since $\mathbb X^1$ is continuously embedded in
$\mathbf H\times L^2(\Omega)$,
\Cref{thm:quantitative-linear-controllability} gives a control satisfying
\eqref{eq:quantitative-control-cost-strong} and driving the corresponding
solution to zero at time $T$. For the same initial datum and control, item
\textup{(ii)} of \Cref{thm:linear-wellposedness} gives a solution
\begin{equation*}
(z,\xi)\in
C([0,T];\mathbb X^1)\cap L^2(0,T;\mathbb X^2).
\end{equation*}
By uniqueness, both solutions coincide. In particular,
$\Sigma_T(y_0,h)=0$.
\end{proof}

We now compare the nonlinear and linearized solution maps. For small initial
data and controls, their difference at time $T$ is quadratic.

\begin{lem}
\label{lem:quadratic-perturbation-3d}
There exist constants $r_0>0$ and $C>0$, depending only on $\Omega$, $\nu$ and
$\kappa$, such that, for every $T\in(0,1)$, every
$y_0\in\mathbb X^1$, and every
$h\in L^2(0,T;L^2(\omega)\times L^2(\omega))$ satisfying
\begin{equation}
\label{eq:smallness-quadratic-perturbation-3d}
\|y_0\|_{\mathbb X^1}
+
\|h\|_{L^2(0,T;L^2(\omega)\times L^2(\omega))}
\leq r_0,
\end{equation}
the nonlinear solution $S_t(y_0,h)$ is well defined on $[0,T]$ and
\begin{equation}
\label{eq:quadratic-perturbation-3d}
\|S_T(y_0,h)-\Sigma_T(y_0,h)\|_{\mathbb X^1}
\leq
C
\left(
\|y_0\|_{\mathbb X^1}
+
\|h\|_{L^2(0,T;L^2(\omega)\times L^2(\omega))}
\right)^2.
\end{equation}
\end{lem}

\begin{proof}
Let $r_\ast>0$ be given by
\Cref{thm:small-data-strong-wellposedness-3d}. Since
\begin{equation*}
\|\Pi(\chi_\omega h_1e_1)\|_{L^2(0,T;\mathbf H)}
+
\|\chi_\omega h_2\|_{L^2(0,T;L^2)}
\leq
\sqrt2\,
\|h\|_{L^2(0,T;L^2(\omega)\times L^2(\omega))},
\end{equation*}
we choose $r_0>0$, independently of $T$, such that
$(1+\sqrt2)r_0\leq r_\ast$. Condition
\eqref{eq:smallness-quadratic-perturbation-3d} then implies the smallness
condition in \Cref{thm:small-data-strong-wellposedness-3d}. Therefore, the
solution to \eqref{eq:nonlinear-boussinesq-3d} is well defined on $[0,T]$ and
satisfies
\begin{equation}
\label{eq:nonlinear-ZT-bound-3d}
\|(u,\theta)\|_{C([0,T];\mathbb X^1)}
+
\|(u,\theta)\|_{L^2(0,T;\mathbb X^2)}
\leq
C
\left(
\|y_0\|_{\mathbb X^1}
+
\|h\|_{L^2(0,T;L^2(\omega)\times L^2(\omega))}
\right).
\end{equation}

Set $w:=u-z$ and $\eta:=\theta-\xi$. Then $(w,\eta)$ satisfies
\begin{equation}
\label{eq:difference-system-nonlinear-linear-3d}
\begin{cases}
w_t+\nu\mathbf A w
=
\Pi(\eta e_3)-\mathbf B(u,u)
& t\in(0,T),\\
\eta_t+\kappa A\eta
=
-u\cdot\nabla\theta
& t\in(0,T),\\
(w,\eta)(0)=(0,0).
\end{cases}
\end{equation}
Applying item \textup{(ii)} of \Cref{thm:linear-wellposedness} to
\eqref{eq:difference-system-nonlinear-linear-3d}, we obtain
\begin{equation}
\label{eq:difference-linear-strong-bound}
\|(w,\eta)\|_{C([0,T];\mathbb X^1)}
+
\|(w,\eta)\|_{L^2(0,T;\mathbb X^2)}
\leq
C
\left(
\|\mathbf B(u,u)\|_{L^2(0,T;\mathbf H)}
+
\|u\cdot\nabla\theta\|_{L^2(0,T;L^2)}
\right).
\end{equation}

Using items \textup{(i)} and \textup{(ii)} of
\Cref{lem:basic-estimates-3d}, we have
\begin{equation*}
\|\mathbf B(u,u)\|_{\mathbf H}
\leq
C\|u\|_{L^6}\|\nabla u\|_{L^3}
\leq
C\|u\|_{\mathbf V}^{3/2}
\|\mathbf A u\|_{\mathbf H}^{1/2}.
\end{equation*}
Hence, by Hölder's inequality and since $T\in(0,1)$, we have
\begin{equation}
\label{eq:Buu-L2H-bound-3d}
\|\mathbf B(u,u)\|_{L^2(0,T;\mathbf H)}
\leq
C
\|u\|_{C([0,T];\mathbf V)}^{3/2}
\|\mathbf A u\|_{L^2(0,T;\mathbf H)}^{1/2}.
\end{equation}
Similarly, using items \textup{(i)} and \textup{(iii)} of
\Cref{lem:basic-estimates-3d}, we obtain
\begin{equation}
\label{eq:utheta-L2-bound-3d}
\|u\cdot\nabla\theta\|_{L^2(0,T;L^2)}
\leq
C
\|u\|_{C([0,T];\mathbf V)}
\|\theta\|_{C([0,T];H^1_0)}^{1/2}
\|A\theta\|_{L^2(0,T;L^2)}^{1/2}.
\end{equation}

Combining \eqref{eq:Buu-L2H-bound-3d}--\eqref{eq:utheta-L2-bound-3d} with
\eqref{eq:nonlinear-ZT-bound-3d}, we get
\begin{equation}
\label{eq:nonlinear-forcing-quadratic-bound-3d}
\|\mathbf B(u,u)\|_{L^2(0,T;\mathbf H)}
+
\|u\cdot\nabla\theta\|_{L^2(0,T;L^2)}
\leq
C
\left(
\|y_0\|_{\mathbb X^1}
+
\|h\|_{L^2(0,T;L^2(\omega)\times L^2(\omega))}
\right)^2.
\end{equation}
Thus, from \eqref{eq:difference-linear-strong-bound},
\begin{equation*}
\|(w,\eta)\|_{C([0,T];\mathbb X^1)}
\leq
C
\left(
\|y_0\|_{\mathbb X^1}
+
\|h\|_{L^2(0,T;L^2(\omega)\times L^2(\omega))}
\right)^2.
\end{equation*}
Since
$S_T(y_0,h)-\Sigma_T(y_0,h)=(w(T),\eta(T))$, this proves
\eqref{eq:quadratic-perturbation-3d}.
\end{proof}

\subsection{The time-iteration argument in dimension three}

We now combine \Cref{cor:quantitative-control-cost-strong} with
\Cref{lem:quadratic-perturbation-3d} to prove the nonlinear controllability
result in dimension three.

\begin{prop}
\label{prop:time-iteration-3d}
Let $T_0\in(0,1)$ be as in
\Cref{cor:quantitative-control-cost-strong}. For every $T\in(0,T_0)$, there
exists $\delta_T>0$ such that, for every
$y_0=(u_0,\theta_0)\in\mathbb X^1$ satisfying
\begin{equation*}
\|y_0\|_{\mathbb X^1}\leq\delta_T,
\end{equation*}
there exists a control
\begin{equation*}
h=(h_1,h_2)\in L^2(0,T;L^2(\omega)\times L^2(\omega))
\end{equation*}
such that the corresponding controlled solution $(u,\theta)$ of
\eqref{eq:nonlinear-boussinesq-3d} is well defined on $[0,T]$ and satisfies
\begin{equation*}
u(T)=0,
\qquad
\theta(T)=0.
\end{equation*}
Moreover, the control can be chosen so that
\begin{equation}
\label{eq:nonlinear-control-cost-3d}
\|h\|_{L^2(0,T;L^2(\omega)\times L^2(\omega))}
\leq
C\exp\left(\frac{C}{T}\right)
\|y_0\|_{\mathbb X^1},
\end{equation}
where $C>0$ is independent of $T$ and $y_0$.
\end{prop}

\begin{proof}
The proof is divided into four steps. Set $\rho:=2^{-1/2}$ and, for a fixed
$T\in(0,T_0)$, define
\begin{equation*}
\tau_j:=T(1-\rho)\rho^j,
\qquad
j\in\mathbb N,
\end{equation*}
and
\begin{equation}
\label{eq:def_sjs_Ijs}
s_0:=0,
\qquad
s_{j+1}:=s_j+\tau_j,
\qquad
I_j:=(s_j,s_{j+1}).
\end{equation}
Note that $\sum_{j=0}^{\infty}\tau_j=T$.

\smallskip
\textit{-- Step 1: A linear control on each time interval.}
We construct inductively a sequence of states
$(Y_j)_{j\in\mathbb N}\subset\mathbb X^1$ and controls
\begin{equation*}
h_j\in L^2(0,\tau_j;L^2(\omega)\times L^2(\omega)).
\end{equation*}
Set $Y_0:=y_0$ and assume that $Y_j\in\mathbb X^1$ has been constructed. If
$Y_j=0$, we take $h_j=0$ and set $Y_{j+1}=0$. Otherwise, since $\tau_j<T<T_0$,
\Cref{cor:quantitative-control-cost-strong} gives a control $h_j$ such that
\begin{equation}
\label{eq:linear-step-j-3d}
\Sigma_{\tau_j}(Y_j,h_j)=0
\qquad\text{and}\qquad
\|h_j\|_{L^2(0,\tau_j;L^2(\omega)\times L^2(\omega))}
\leq
C_L\exp\left(\frac{C_L}{\tau_j}\right)
\|Y_j\|_{\mathbb X^1},
\end{equation}
for some $C_L>0$ independent of $j$, $T$, and $Y_j$.

Let $r_0>0$ be the radius in
\Cref{lem:quadratic-perturbation-3d}. Whenever
\begin{equation}
\label{eq:local-smallness-slice-3d}
\|Y_j\|_{\mathbb X^1}
+
\|h_j\|_{L^2(0,\tau_j;L^2(\omega)\times L^2(\omega))}
\leq r_0,
\end{equation}
we define
\begin{equation}
\label{eq:Yj-plus-one-definition}
Y_{j+1}:=S_{\tau_j}(Y_j,h_j).
\end{equation}
The construction is continued as long as
\eqref{eq:local-smallness-slice-3d} holds. In Step 3, we choose the initial
datum sufficiently small and verify inductively that this condition is
satisfied for every $j\in\mathbb N$.

\smallskip
\textit{-- Step 2: Nonlinear contraction.}
Assume that the construction has been carried out up to the index $j$. Let
$C_N>0$ be the constant provided by
\Cref{lem:quadratic-perturbation-3d}. Since
$\Sigma_{\tau_j}(Y_j,h_j)=0$, estimate
\eqref{eq:quadratic-perturbation-3d} and
\eqref{eq:Yj-plus-one-definition} give
\begin{equation*}
\|Y_{j+1}\|_{\mathbb X^1}
\leq
C_N
\left(
\|Y_j\|_{\mathbb X^1}
+
\|h_j\|_{L^2(0,\tau_j;L^2(\omega)\times L^2(\omega))}
\right)^2.
\end{equation*}
We choose $K_0\geq\max\{1,C_L\}$, independently of $\tau$, sufficiently large
so that
\begin{equation}
\label{eq:K0-choice-3d}
C_N
\left(
1+C_L\exp\left(\frac{C_L}{\tau}\right)
\right)^2
\leq
K_0\exp\left(\frac{K_0}{\tau}\right),
\qquad
\tau>0.
\end{equation}
Combining \eqref{eq:linear-step-j-3d} and
\eqref{eq:K0-choice-3d}, we obtain
\begin{equation}
\label{eq:one-step-recursion-3d}
\|Y_{j+1}\|_{\mathbb X^1}
\leq
K_0\exp\left(\frac{K_0}{\tau_j}\right)
\|Y_j\|_{\mathbb X^1}^2.
\end{equation}

If some $Y_j$ vanishes, the construction is completed by taking all the
remaining controls equal to zero. Hence, for the following estimates, we may
assume that $Y_j\neq0$. Taking logarithms in
\eqref{eq:one-step-recursion-3d} and dividing by $2^{j+1}$, we get
\begin{equation}
\label{eq:log-recursion-3d}
\frac{\log\|Y_{j+1}\|_{\mathbb X^1}}{2^{j+1}}
\leq
\frac{\log\|Y_j\|_{\mathbb X^1}}{2^j}
+
\frac{\log K_0}{2^{j+1}}
+
\frac{K_0}{2^{j+1}\tau_j}.
\end{equation}
Set
\begin{equation*}
A_j:=
\frac{\log\|Y_j\|_{\mathbb X^1}}{2^j}.
\end{equation*}
Summing \eqref{eq:log-recursion-3d} from $j=0$ to $j=n-1$, we obtain
\begin{equation}
\label{eq:sum_An}
A_n
\leq
A_0
+
\sum_{j=0}^{n-1}\frac{\log K_0}{2^{j+1}}
+
\sum_{j=0}^{n-1}\frac{K_0}{2^{j+1}\tau_j}.
\end{equation}
Since $A_0=\log\|y_0\|_{\mathbb X^1}$,
$\sum_{j=0}^{\infty}2^{-(j+1)}=1$, and
\begin{equation*}
\begin{aligned}
\sum_{j=0}^{\infty}\frac{1}{2^{j+1}\tau_j}
&=
\frac{1}{2T(1-\rho)}
\sum_{j=0}^{\infty}
\left(\frac{\rho^{-1}}{2}\right)^j
\\
&=
\frac{1}{2T(1-\rho)}
\sum_{j=0}^{\infty}\rho^j
=
\frac{1}{2T(1-\rho)^2},
\end{aligned}
\end{equation*}
we deduce from \eqref{eq:sum_An} that
\begin{equation}
\label{eq:log-bound-3d}
\frac{\log\|Y_n\|_{\mathbb X^1}}{2^n}
\leq
\log\|y_0\|_{\mathbb X^1}+C_T,
\end{equation}
where
\begin{equation*}
C_T:=
\log K_0+\frac{K_0}{2T(1-\rho)^2}.
\end{equation*}

Let
\begin{equation}
\label{eq:M0-definition-3d}
M_0:=
\max\left\{0,\log\left(\frac{2}{r_0}\right)\right\},
\end{equation}
and define
\begin{equation}
\label{eq:MT-definition-3d}
M_T:=
\log2+\frac{K_0}{\tau_0}+M_0.
\end{equation}
Then $e^{-M_T}\leq r_0/2$. Moreover, for every $j\in\mathbb N$,
\begin{equation}
\label{eq:MT-choice-3d}
(2^j-1)M_T
\geq
j\log2
+
K_0\left(\frac1{\tau_j}-\frac1{\tau_0}\right).
\end{equation}
Indeed, $j\leq2^j-1$ and
\begin{equation*}
\frac1{\tau_j}-\frac1{\tau_0}
=
\frac{1}{\tau_0}(\rho^{-j}-1)
=
\frac{1}{\tau_0}(2^{j/2}-1)
\leq
\frac{1}{\tau_0}(2^j-1).
\end{equation*}

Assume now that
\begin{equation}
\label{eq:smallness-initial-3d}
\|y_0\|_{\mathbb X^1}
\leq
e^{-M_T-C_T}.
\end{equation}
It follows from \eqref{eq:log-bound-3d} that
\begin{equation}
\label{eq:double-exp-3d}
\|Y_j\|_{\mathbb X^1}
\leq
e^{-M_T2^j},
\qquad
j\in\mathbb N.
\end{equation}
In particular, $Y_j\to0$ in $\mathbb X^1$ as $j\to+\infty$.

\smallskip
\textit{-- Step 3: Control cost and validity of the construction.}
From \eqref{eq:linear-step-j-3d} and \eqref{eq:log-bound-3d}, we have
\begin{equation*}
\begin{aligned}
\log
\frac{
\|h_j\|_{L^2(0,\tau_j;L^2(\omega)\times L^2(\omega))}
}{
\|y_0\|_{\mathbb X^1}
}
&\leq
\log C_L
+
\frac{C_L}{\tau_j}
+
(2^j-1)\log\|y_0\|_{\mathbb X^1}
+
2^jC_T.
\end{aligned}
\end{equation*}
Using \eqref{eq:smallness-initial-3d}, we get
\begin{equation*}
\begin{aligned}
\log
\frac{
\|h_j\|_{L^2(0,\tau_j;L^2(\omega)\times L^2(\omega))}
}{
\|y_0\|_{\mathbb X^1}
}
&\leq
\log C_L
+
\frac{C_L}{\tau_j}
-
(2^j-1)M_T
+
C_T.
\end{aligned}
\end{equation*}
Then, from \eqref{eq:MT-choice-3d}, we deduce
\begin{equation*}
\begin{aligned}
\log
\frac{
\|h_j\|_{L^2(0,\tau_j;L^2(\omega)\times L^2(\omega))}
}{
\|y_0\|_{\mathbb X^1}
}
&\leq
\log C_L
+
\frac{C_L}{\tau_j}
+
C_T
-
j\log2
-
K_0\left(\frac1{\tau_j}-\frac1{\tau_0}\right)
\\
&\leq
\log C_L
+
C_T
-
j\log2
+
\frac{K_0}{\tau_0},
\end{aligned}
\end{equation*}
where we have used $K_0\geq C_L$. Therefore,
\begin{equation}
\label{eq:hj-geometric-bound-3d}
\|h_j\|_{L^2(0,\tau_j;L^2(\omega)\times L^2(\omega))}
\leq
2^{-j}C_L
\exp\left(
C_T+\frac{K_0}{\tau_0}
\right)
\|y_0\|_{\mathbb X^1}.
\end{equation}
Since $\tau_0=T(1-\rho)$ and $C_T\leq C/T$, we deduce
\begin{equation}
\label{eq:sum-controls-cost-3d}
\sum_{j=0}^{\infty}
\|h_j\|_{L^2(0,\tau_j;L^2(\omega)\times L^2(\omega))}
\leq
C\exp\left(\frac{C}{T}\right)
\|y_0\|_{\mathbb X^1},
\end{equation}
where $C>0$ is independent of $T$.

We choose $\delta_T>0$ sufficiently small so that
\begin{equation}
\label{eq:deltaT-choice-3d}
\delta_T\leq e^{-M_T-C_T}
\qquad\text{and}\qquad
C_L
\exp\left(
C_T+\frac{K_0}{\tau_0}
\right)
\delta_T
\leq
\frac{r_0}{2}.
\end{equation}
If $\|y_0\|_{\mathbb X^1}\leq\delta_T$, then
\eqref{eq:double-exp-3d}, \eqref{eq:hj-geometric-bound-3d}, and the definitions
of $M_0$ and $M_T$ give
\begin{equation}
\label{eq:small_construction}
\|Y_j\|_{\mathbb X^1}
\leq
\frac{r_0}{2},
\qquad
\|h_j\|_{L^2(0,\tau_j;L^2(\omega)\times L^2(\omega))}
\leq
2^{-j}\frac{r_0}{2}
\leq
\frac{r_0}{2}.
\end{equation}
These estimates are valid at every index for which the construction is
defined. At $j=0$, they imply
\eqref{eq:local-smallness-slice-3d}, and hence $Y_1$ is well defined. Assuming
that the construction has been carried out up to the index $j$, the same
estimates imply \eqref{eq:local-smallness-slice-3d} at that index and allow us
to define $Y_{j+1}$. Consequently, the construction can be continued
inductively for every $j\in\mathbb N$.

\smallskip
\textit{-- Step 4: Conclusion.}
For $t\in I_j$, we define
\begin{equation*}
h(t):=h_j(t-s_j).
\end{equation*}
Then
\begin{equation*}
\|h\|_{L^2(0,T;L^2(\omega)\times L^2(\omega))}^2
=
\sum_{j=0}^{\infty}
\|h_j\|_{L^2(0,\tau_j;L^2(\omega)\times L^2(\omega))}^2.
\end{equation*}
Hence, from \eqref{eq:sum-controls-cost-3d},
\begin{equation}
\label{eq:global-nonlinear-control-cost-3d}
\|h\|_{L^2(0,T;L^2(\omega)\times L^2(\omega))}
\leq
C\exp\left(\frac{C}{T}\right)
\|y_0\|_{\mathbb X^1}.
\end{equation}

We now concatenate the corresponding nonlinear solutions. On each interval
$I_j$, we solve the nonlinear system with initial datum $Y_j$ at time $s_j$
and control $h_j(\cdot-s_j)$. By
\eqref{eq:local-smallness-slice-3d}, these solutions are well defined and
unique. They can therefore be concatenated to obtain a solution $(u,\theta)$
on $[0,T)$ satisfying
\begin{equation*}
(u(s_j),\theta(s_j))=Y_j,
\qquad
j\in\mathbb N.
\end{equation*}

The local well-posedness estimate on each interval $I_j$ gives
\begin{equation*}
\|(u,\theta)\|_{L^2(I_j;\mathbb X^2)}
\leq
C
\left(
\|Y_j\|_{\mathbb X^1}
+
\|h_j\|_{L^2(0,\tau_j;L^2(\omega)\times L^2(\omega))}
\right).
\end{equation*}
Consequently, by \eqref{eq:double-exp-3d} and
\eqref{eq:hj-geometric-bound-3d},
$
\sum_{j=0}^{\infty}
\|(u,\theta)\|_{L^2(I_j;\mathbb X^2)}^2
<+\infty,
$
and hence $(u,\theta)\in L^2(0,T;\mathbb X^2)$.

Let $t\in I_j$. The same local well-posedness estimate gives
\begin{equation*}
\|(u(t),\theta(t))\|_{\mathbb X^1}
\leq
C
\left(
\|Y_j\|_{\mathbb X^1}
+
\|h_j\|_{L^2(0,\tau_j;L^2(\omega)\times L^2(\omega))}
\right).
\end{equation*}
The right-hand side tends to zero as $j\to\infty$ by
\eqref{eq:double-exp-3d} and \eqref{eq:hj-geometric-bound-3d}. Since
$s_j\uparrow T$, we conclude that
\begin{equation*}
(u(t),\theta(t))\to(0,0)
\qquad
\text{in }\mathbb X^1
\quad\text{as }t\uparrow T.
\end{equation*}
Defining $(u(T),\theta(T))=(0,0)$, we obtain
\begin{equation*}
(u,\theta)
\in
C([0,T];\mathbb X^1)
\cap
L^2(0,T;\mathbb X^2).
\end{equation*}
Together with \eqref{eq:global-nonlinear-control-cost-3d}, this proves the
result.
\end{proof}

\subsection{Remarks on the two-dimensional case}

We now consider the case $N=2$. In this dimension, no control acts directly
on the velocity equation. For $y_0=(u_0,\theta_0)\in\mathbb X^0$ and
$h_2\in L^2(0,T;L^2(\omega))$, we consider the nonlinear system
\begin{equation}
\label{eq:nonlinear-boussinesq-2d}
\begin{cases}
u_t+\nu\mathbf A u+\mathbf B(u,u)
=
\Pi(\theta e_2)
& t\in(0,T),\\
\theta_t+\kappa A\theta+u\cdot\nabla\theta
=
\chi_\omega h_2
& t\in(0,T),\\
(u(0),\theta(0))=(u_0,\theta_0),
\end{cases}
\end{equation}
and the corresponding linearized system
\begin{equation}
\label{eq:linear-boussinesq-2d}
\begin{cases}
z_t+\nu\mathbf A z
=
\Pi(\xi e_2)
& t\in(0,T),\\
\xi_t+\kappa A\xi
=
\chi_\omega h_2
& t\in(0,T),\\
(z(0),\xi(0))=(u_0,\theta_0).
\end{cases}
\end{equation}
We denote their solution maps by
$S_t^{(2)}(y_0,h_2):=(u(t),\theta(t))$ and
$\Sigma_t^{(2)}(y_0,h_2):=(z(t),\xi(t))$.

By \Cref{rem:nonlinear-wellposedness-2d}, system
\eqref{eq:nonlinear-boussinesq-2d} has a unique global weak solution in
$C([0,T];\mathbb X^0)\cap L^2(0,T;\mathbb X^1)$. Moreover, for
$T\in(0,1)$, the cancellations in \eqref{eq:cancelations_2d} and Gronwall's
inequality give
\begin{equation}
\label{eq:nonlinear-energy-bound-2d}
\|(u,\theta)\|_{C([0,T];\mathbb X^0)}
+
\|(u,\theta)\|_{L^2(0,T;\mathbb X^1)}
\leq
C
\left(
\|y_0\|_{\mathbb X^0}
+
\|h_2\|_{L^2(0,T;L^2(\omega))}
\right),
\end{equation}
where $C>0$ is independent of $T$.

We first estimate the difference between the two solution maps in the energy
space.

\begin{lem}
\label{lem:quadratic-perturbation-2d}
There exists $C>0$, depending only on $\Omega$, $\nu$, and $\kappa$, such
that, for every $T\in(0,1)$, $y_0\in\mathbb X^0$, and
$h_2\in L^2(0,T;L^2(\omega))$,
\begin{equation}
\label{eq:quadratic-perturbation-2d}
\left\|
S_T^{(2)}(y_0,h_2)-\Sigma_T^{(2)}(y_0,h_2)
\right\|_{\mathbb X^0}
\leq
C
\left(
\|y_0\|_{\mathbb X^0}
+
\|h_2\|_{L^2(0,T;L^2(\omega))}
\right)^2.
\end{equation}
\end{lem}

\begin{proof}
Let $(u,\theta)$ and $(z,\xi)$ be the solutions of
\eqref{eq:nonlinear-boussinesq-2d} and
\eqref{eq:linear-boussinesq-2d}, respectively. Setting $w:=u-z$ and
$\eta:=\theta-\xi$, we have
\begin{equation}
\label{eq:difference-system-2d}
\begin{cases}
w_t+\nu\mathbf A w
=
\Pi(\eta e_2)-\mathbf B(u,u)
& t\in(0,T),\\
\eta_t+\kappa A\eta
=
-u\cdot\nabla\theta
& t\in(0,T),\\
(w,\eta)(0)=(0,0).
\end{cases}
\end{equation}
Taking the scalar product of the first equation with $w$ in $\mathbf H$ and
of the second equation with $\eta$ in $L^2(\Omega)$, we obtain
\begin{equation}
\label{eq:difference-energy-identity-2d}
\begin{aligned}
\frac12\frac{\d}{\dt}
\left(
\|w\|_{\mathbf H}^2+\|\eta\|_{L^2}^2
\right)
&+
\nu\|w\|_{\mathbf V}^2
+
\kappa\|\eta\|_{H^1_0}^2
\\
&=
\bigl(\Pi(\eta e_2),w\bigr)_{\mathbf H}
-
\langle\mathbf B(u,u),w\rangle_{\mathbf V',\mathbf V}
-
\int_\Omega(u\cdot\nabla\theta)\eta\dx.
\end{aligned}
\end{equation}
The coupling term is bounded by
$C(\|w\|_{\mathbf H}^2+\|\eta\|_{L^2}^2)$. On the other hand, the
two-dimensional interpolation inequality
$\|u\|_{L^4}^2\leq C\|u\|_{\mathbf H}\|u\|_{\mathbf V}$ gives
\begin{equation*}
\begin{aligned}
\left|
\langle\mathbf B(u,u),w\rangle_{\mathbf V',\mathbf V}
\right|
&\leq
C\|u\|_{L^4}^2\|w\|_{\mathbf V}
\leq
\frac{\nu}{4}\|w\|_{\mathbf V}^2
+
C\|u\|_{\mathbf H}^2\|u\|_{\mathbf V}^2.
\end{aligned}
\end{equation*}
Since $\nabla\cdot u=0$ and $u=0$ on $\partial\Omega$, integration by parts
also gives
\begin{equation*}
\begin{aligned}
\left|
\int_\Omega(u\cdot\nabla\theta)\eta\dx
\right|
&=
\left|
\int_\Omega(u\cdot\nabla\eta)\theta\dx
\right| \leq
\|u\|_{L^4}\|\theta\|_{L^4}\|\nabla\eta\|_{L^2}
\\
&\leq
\frac{\kappa}{4}\|\eta\|_{H^1_0}^2
+
C\|u\|_{L^4}^2\|\theta\|_{L^4}^2.
\end{aligned}
\end{equation*}
It follows from \eqref{eq:difference-energy-identity-2d} that, for some
$c_0>0$,
\begin{equation}
\label{eq:difference-energy-inequality-2d}
\begin{aligned}
\frac{\d}{\dt}
\left(
\|w\|_{\mathbf H}^2+\|\eta\|_{L^2}^2
\right)
+
c_0
\left(
\|w\|_{\mathbf V}^2+\|\eta\|_{H^1_0}^2
\right) 
&\leq
C
\left(
\|w\|_{\mathbf H}^2+\|\eta\|_{L^2}^2
\right)
+
C\|u\|_{\mathbf H}^2\|u\|_{\mathbf V}^2
\\
&\quad
+
C\|u\|_{L^4}^2\|\theta\|_{L^4}^2.
\end{aligned}
\end{equation}

Set $R:=
\|y_0\|_{\mathbb X^0}
+
\|h_2\|_{L^2(0,T;L^2(\omega))}.
$
Estimate \eqref{eq:nonlinear-energy-bound-2d} yields
\begin{equation*}
\int_0^T
\|u\|_{\mathbf H}^2\|u\|_{\mathbf V}^2\dt
\leq
\|u\|_{L^\infty(0,T;\mathbf H)}^2
\|u\|_{L^2(0,T;\mathbf V)}^2
\leq
CR^4.
\end{equation*}
Similarly, using
$\|\theta\|_{L^4}^2
\leq C\|\theta\|_{L^2}\|\theta\|_{H^1_0}$, we obtain
\begin{equation*}
\begin{aligned}
\int_0^T
\|u\|_{L^4}^2\|\theta\|_{L^4}^2\dt
&\leq
C
\|u\|_{L^\infty(0,T;\mathbf H)}
\|\theta\|_{L^\infty(0,T;L^2)}
\\
&\quad\times
\|u\|_{L^2(0,T;\mathbf V)}
\|\theta\|_{L^2(0,T;H^1_0)}
\leq
CR^4.
\end{aligned}
\end{equation*}
Since $(w,\eta)(0)=(0,0)$ and $T\in(0,1)$, Gronwall's inequality applied to
\eqref{eq:difference-energy-inequality-2d} gives
\begin{equation*}
\|(w,\eta)\|_{C([0,T];\mathbb X^0)}^2
\leq
CR^4.
\end{equation*}
In particular, $\|(w(T),\eta(T))\|_{\mathbb X^0}\leq CR^2$, which proves
\eqref{eq:quadratic-perturbation-2d}.
\end{proof}

We can now apply the time-iteration argument in $\mathbb X^0$.

\begin{prop}
\label{prop:time-iteration-2d}
Assume that $N=2$. There exist $T_0\in(0,1)$ and $C>0$ such that, for every
$T\in(0,T_0)$, there exists $\delta_T>0$ with the following property: for
every $y_0\in\mathbb X^0$ with $\|y_0\|_{\mathbb X^0}\leq\delta_T$, there
exists a control $h_2\in L^2(0,T;L^2(\omega))$ such that
\begin{equation*}
S_T^{(2)}(y_0,h_2)=0.
\end{equation*}
Moreover,
\begin{equation}
\label{eq:nonlinear-control-cost-2d}
\|h_2\|_{L^2(0,T;L^2(\omega))}
\leq
C\exp\left(\frac{C}{T}\right)\|y_0\|_{\mathbb X^0},
\end{equation}
where $C>0$ is independent of $T$ and $y_0$.
\end{prop}

\begin{proof}
For every $\tau\in(0,T_0)$ and $Y\in\mathbb X^0$,
\Cref{thm:quantitative-linear-controllability}, applied with $N=2$, gives a
control $h_2\in L^2(0,\tau;L^2(\omega))$ such that
\begin{equation}
\label{eq:linear-control-step-2d}
\Sigma_\tau^{(2)}(Y,h_2)=0,
\qquad
\|h_2\|_{L^2(0,\tau;L^2(\omega))}
\leq
C_L\exp\left(\frac{C_L}{\tau}\right)\|Y\|_{\mathbb X^0},
\end{equation}
for some $C_L>0$ independent of $\tau$ and $Y$.

We use the same partition of $(0,T)$ as in the proof of
\Cref{prop:time-iteration-3d}. Namely, let $\rho=2^{-1/2}$,
$\tau_j:=T(1-\rho)\rho^j$, $s_0:=0$, and
$s_{j+1}:=s_j+\tau_j$. Starting from $Y_0:=y_0$, we choose
$h_2^j\in L^2(0,\tau_j;L^2(\omega))$ satisfying
\eqref{eq:linear-control-step-2d} and set
$Y_{j+1}:=S_{\tau_j}^{(2)}(Y_j,h_2^j)$. The global well-posedness of
\eqref{eq:nonlinear-boussinesq-2d} guarantees that $Y_{j+1}$ is well defined
for every $j\in\mathbb N$.

Since $\Sigma_{\tau_j}^{(2)}(Y_j,h_2^j)=0$,
\Cref{lem:quadratic-perturbation-2d} gives
\begin{equation}
\label{eq:one-step-recursion-2d}
\|Y_{j+1}\|_{\mathbb X^0}
\leq
K_0\exp\left(\frac{K_0}{\tau_j}\right)
\|Y_j\|_{\mathbb X^0}^2
\end{equation}
for some $K_0\geq\max\{1,C_L\}$ independent of $j$ and $T$.

The remaining estimates are the same as in Steps 2 and 3 of the proof of
\Cref{prop:time-iteration-3d}, with $\mathbb X^1$ replaced by
$\mathbb X^0$. Those estimates only use the linear control bound
\eqref{eq:linear-control-step-2d} and the quadratic recurrence
\eqref{eq:one-step-recursion-2d}. Consequently, there exists $\delta_T>0$
such that, if $\|y_0\|_{\mathbb X^0}\leq\delta_T$, then
$Y_j\to0$ in $\mathbb X^0$ and
\begin{equation}
\label{eq:sum-controls-2d}
\sum_{j=0}^{\infty}
\|h_2^j\|_{L^2(0,\tau_j;L^2(\omega))}
\leq
C\exp\left(\frac{C}{T}\right)\|y_0\|_{\mathbb X^0}.
\end{equation}

For $t\in(s_j,s_{j+1})$, define $h_2(t):=h_2^j(t-s_j)$. Then
$h_2\in L^2(0,T;L^2(\omega))$, and
\eqref{eq:sum-controls-2d} gives
\eqref{eq:nonlinear-control-cost-2d}. By uniqueness, the solutions on the
successive intervals can be concatenated. Finally, for
$t\in(s_j,s_{j+1})$, estimate
\eqref{eq:nonlinear-energy-bound-2d} gives
\begin{equation*}
\left\|
S_{t-s_j}^{(2)}(Y_j,h_2^j)
\right\|_{\mathbb X^0}
\leq
C
\left(
\|Y_j\|_{\mathbb X^0}
+
\|h_2^j\|_{L^2(0,\tau_j;L^2(\omega))}
\right).
\end{equation*}
The right-hand side tends to zero as $j\to\infty$. Since
$s_j\uparrow T$, the concatenated solution converges to zero in
$\mathbb X^0$ as $t\uparrow T$. This proves the result.
\end{proof}

\subsection{Proof of the main result}

We now conclude the proof of the main theorem.

\begin{proof}[Proof of \Cref{thm:main-nonlinear-controllability}]
For $N=3$, the result follows from
\Cref{prop:time-iteration-3d}, since
$\mathbb X^1=\mathbf V\times H^1_0(\Omega)$. It is enough to take
$v=(h_1,0,0)$ and $v_0=h_2$.

For $N=2$, \Cref{prop:time-iteration-2d} gives the result in the larger space
$\mathbb X^0$. Since $\mathbb X^1\hookrightarrow\mathbb X^0$, it applies,
after adjusting $\delta_T$ if necessary, to the initial data considered in
the theorem. For such data, the weak solution has the additional regularity
given in \Cref{rem:nonlinear-wellposedness-2d} and coincides with the strong
solution by uniqueness. In this case, we take $v=(0,0)$ and $v_0=h_2$.

Finally, estimates \eqref{eq:nonlinear-control-cost-3d} and
\eqref{eq:nonlinear-control-cost-2d} give
\eqref{eq:main-nonlinear-control-cost} in their respective dimensions.
\end{proof}

\bibliographystyle{alpha}
\small{\bibliography{bib_bou_LR}}

\bigskip

\begin{flushleft}

\textbf{Víctor Hernández-Santamaría}\\
Departamento de Matem\'aticas, Facultad de Ciencias\\
Universidad Nacional Autónoma de México \\
Circuito Exterior, C.U.\\
04510, Coyoacán, CDMX, Mexico\\
\texttt{victor.santamaria@ciencias.unam.mx}

\end{flushleft}

\end{document}